\documentclass[12pt]{article}

\usepackage{epsfig}
\usepackage{latexsym}
\usepackage{amsmath}
\usepackage{amsfonts}
\usepackage{amssymb}
\usepackage{graphicx}%
\usepackage{algorithmic}
\usepackage{color}

\makeatletter

\newcommand{\Rmnum}[1]{\expandafter\@slowromancap\romannumeral #1@}
\makeatother


\begin{document}
	
	\pagestyle{myheadings} \markright{\sc Packing branchings under cardinality constraints on their root sets \hfill} \thispagestyle{empty}
	
	\newtheorem{theorem}{Theorem}[section]
	\newtheorem{corollary}[theorem]{Corollary}
	\newtheorem{definition}[theorem]{Definition}
	\newtheorem{guess}[theorem]{Conjecture}
	\newtheorem{claim}[theorem]{Claim}
	\newtheorem{problem}[theorem]{Problem}
	\newtheorem{question}[theorem]{Question}
	\newtheorem{lemma}[theorem]{Lemma}
	\newtheorem{proposition}[theorem]{Proposition}
	\newtheorem{fact}[theorem]{Fact}
	\newtheorem{acknowledgement}[theorem]{Acknowledgement}
	\newtheorem{algorithm}[theorem]{Algorithm}
	\newtheorem{axiom}[theorem]{Axiom}
	\newtheorem{case}[theorem]{Case}
	\newtheorem{conclusion}[theorem]{Conclusion}
	\newtheorem{condition}[theorem]{Condition}
	\newtheorem{conjecture}[theorem]{Conjecture}
	\newtheorem{criterion}[theorem]{Criterion}
	\newtheorem{example}[theorem]{Example}
	\newtheorem{exercise}[theorem]{Exercise}
	\newtheorem{notation}[theorem]{Notation}
	\newtheorem{observation}[theorem]{Observation}
	\newtheorem{solution}[theorem]{Solution}
	\newtheorem{summary}[theorem]{Summary}
	
	\newtheorem{thm}[theorem]{Theorem}
	\newtheorem{prop}[theorem]{Proposition}
	\newtheorem{defn}[theorem]{Definition}

\newtheorem{lem}[theorem]{Lemma}
\newtheorem{con}[theorem]{Conjecture}
\newtheorem{cor}[theorem]{Corollary}

\newenvironment{proof}{\noindent {\bf
		Proof.}}{\rule{3mm}{3mm}\par\medskip}
\newcommand{\remark}{\medskip\par\noindent {\bf Remark.~~}}
\newcommand{\pp}{{\it p.}}
\newcommand{\de}{\em}

\newcommand{\g}{\mathrm{g}}

\newcommand{\qf}{Q({\cal F},s)}
\newcommand{\qff}{Q({\cal F}',s)}
\newcommand{\qfff}{Q({\cal F}'',s)}
\newcommand{\f}{{\cal F}}
\newcommand{\ff}{{\cal F}'}
\newcommand{\fff}{{\cal F}''}
\newcommand{\fs}{{\cal F},s}
\newcommand{\cs}{\chi'_s(G)}

\newcommand{\G}{\Gamma}
\newcommand{\wrt}{with respect to }
\newcommand{\mad}{{\rm mad}}
\newcommand{\col}{{\rm col}}
\newcommand{\gcol}{{\rm gcol}}

%%%%%%%%%%%%%%%%%%%%%%%%%%%%%% User specified LaTeX commands.
%\usepackage{fullpage}
\newcommand*{\ch}{{\rm ch}}
\newcommand*{\ra}{{\rm ran}}
\newcommand{\co}{{\rm col}}
\newcommand{\sco}{{\rm scol}}
\newcommand{\wc}{{\rm wcol}}
\newcommand{\dc}{{\rm dcol}}
\newcommand*{\ar}{{\rm arb}}
\newcommand*{\ma}{{\rm mad}}
\newcommand{\di}{{\rm dist}}
\newcommand{\tw}{{\rm tw}}
\newcommand{\scol}{{\rm scol}}
\newcommand{\wcol}{{\rm wcol}}
\newcommand{\td}{{\rm td}}
\newcommand{\edp}[2]{#1^{[\natural #2]}}
\newcommand{\epp}[2]{#1^{\natural #2}}
\newcommand*{\ind}{{\rm ind}}
\newcommand{\red}[1]{\textcolor{red}{#1}}

\def\C#1{|#1|}
\def\E#1{|E(#1)|}
\def\V#1{|V(#1)|}
\def\iarb{\Upsilon}
\def\ipac{\nu}
\def\nul{\varnothing}

\newcommand*{\QEDA}{\ensuremath{\blacksquare}}
\newcommand*{\QEDB}{\hfill\ensuremath{\square}}

\title{\Large\bf Packing branchings under cardinality constraints on their root sets}

\author{Hui Gao\thanks{E-mail: {\small\texttt{gaoh1118@yeah.net}}.}\\
	Center for Discrete Mathematics\\
	Fuzhou University\\
	Fuzhou, Fujian 350108, China
	\and
	Daqing Yang\thanks{Corresponding author,  grant number:	NSFC  11871439. E-mail: \small\texttt{dyang@zjnu.edu.cn}.} \\
	Department of Mathematics \\
	Zhejiang Normal University \\
	Jinhua, Zhejiang 321004, China
}

\maketitle

\begin{abstract}
	Edmonds' fundamental theorem on arborescences characterizes the existence of $k$ pairwise arc-disjoint spanning arborescences with prescribed root sets in a digraph. 
	In this paper, we study the problem of packing branchings in digraphs under cardinality constraints on their root sets by  arborescence augmentation.   
	Let	$D=(V+x,A)$ be a digraph, $\mathcal{P}=$ $\{I_{1}, \ldots, I_{l} \}$ be a partition of $[k]$,  $c_{1}, \ldots, c_{l}, c'_{1}, \ldots, c'_{l}$ be nonnegative integers such that $c_{\alpha} \leq c'_{\alpha}$ for $\alpha \in [l]$, $F_{1}, \ldots, F_{k}$ be $k$ arc-disjoint $x$-arborescences in $D$ such that $\sum_{i \in I_{\alpha}}d_{F_{i}}^{+}(x)$  $\leq c'_{\alpha}$ for $\alpha \in [l]$. We give a characterization on when  $F_{1}, \ldots, F_{k}$ can be completed to arc-disjoint spanning $x$-arborescences $F^{*}_{1}, \ldots, F^{*}_{k}$ such that for any $\alpha \in [l]$, $ c_{\alpha} \leq \sum_{i \in I_{\alpha}}d^{+}_{F^{*}_{i}}(x)$ $ \leq c'_{\alpha}$. 	
\end{abstract}

{\em Keywords: Arborescences; Branchings; Submodular functions; Coverings; Packings}

{\em AMS subject classifications.  05B35, 05C40, 05C70}

\section{Introduction}
In this paper, all digraphs can have multiple arcs but not loops. 
For a digraph $D$, denote its vertex set and arc set by $V(D)$ and $E(D)$, respectively. 
Let $\Omega$ be a set, $u \notin \Omega $ and $v \in \Omega$; for simplicity, we write $\Omega+u$ and $\Omega-v$ instead of $\Omega \cup \{ u\}$ and $\Omega \setminus \{v \}$. 
Let $D=(V+x,A)$ be a digraph,  $A_{0} \subseteq A$; denote by $E_{A_{0}}^{+}(x)$ the set of arcs in $A_{0}$ with $x$ as their tail, by $N_{A_{0}}^{+}(x)$ the set of heads of arcs in  $E_{A_{0}}^{+}(x)$; 
for simplicity, we write $E^{+}(x)$ for $E^{+}_{A}(x)$, $N^{+}(x)$ for $N_{A}^{+}(x)$.  
Let $X, Y \subseteq V+x$,  
denote by $[X,Y]_{D}$ the number of arcs in $D$ with their tails in $X$ and heads in $Y$.
Sometimes, instead of $[X,Y]_{D}$,  we  write 
$d_{D}^{+}(X)$ or $d^{-}_{D}(Y)$ when $Y= \overline{X}$; $[x_{0}, Y]_{D}$ when $X=\{x_{0}\}$; $[X,y_{0}]_{D}$ when $Y=\{y_{0} \}$; $[x_{0}, y_{0}]_{D}$ when $X=\{x_{0} \}$ and $Y=\{y_{0}\}$. 
We drop the subscript $D$ in the above notations when $D$ is clear from the context.
For simplicity, we do not distinguish between the arc set $A_0$  of $D$  and the subdigraph of $D$ spanned by $A_0$ itself. We denote the set $\{1,\ldots,k\}$ by $[k]$.

A subdigraph $F$ (it may not be spanning) of $D$ is called an \emph{$r$-arborescence} if its underlying graph is a tree; 
for any $u \in V(F)$,  there is exactly one directed path in $F$ from $r$ to $u$. 
The vertex $r$ is called \emph{root of the arborescence}. A \emph{branching} $B$ in $D$ is a spanning subdigraph each component of which is an arborescence, and the \emph{root set} $R(B)$ of $B$  consists of all roots of its components. Let $c$ be a positive integer. We call $B$ a \emph{$c$-branching, $c^{+}$-branching, $c^{-}$-branching} if $|R(B)|=c$, $|R(B)| \geq c$,  $|R(B)| \leq c$.

The $c$-branching is a directed version of $c$-forests in graphs. The covering  and packing  of graphs by  $c$-forests were first considered by Chen et al.~\cite{CKP}, further studied in \cite{CL}, their extensions and also matroidal version  
 have been studied in \cite{GJLWYZ}.
 The covering  and packing  of digraphs by branchings have been widely studied (cf. \cite{C-6,edmonds,L-31,F-13,frank}),
 and a lot of variations and generalizations have been developed, see the book of Schrijver \cite{Sch} or a recent survey by
 Kamiyama \cite{kam}.

In this paper, we study the problem of packing branchings in  digraphs under cardinality constraints on their root sets by arborescence augmentation. 
We also apply our new results to some classical covering  and packing problems of digraphs by branchings. 
The study of this topic  was initiated   
by the fundamental result of Edmonds (Theorem \ref{16}).  
In the recent years, due to a beautiful extension of Edmonds' classical result by
Kamiyama, Katoh and Takizawa \cite{kam-ka-ta-09} and by Fujishige \cite{fuji-10},  the research 	of this area became particularly active. In the last decade, a series of interesting results
appeared that provide extensions to intersecting bi-set families (first studied by B\'{e}rczi and Frank \cite{BF-1,BF-2}, see also \cite{berczi-16}),  %B\'{e}rczi, Kir\'{a}ly and Kobayashi \cite{berczi-16}), 
to kernel systems (by Leston-Rey and Wakabayashi \cite{le-wa-15}), or under matroid constraints( \cite{durand-nguyen-szigeti-13,fortier-kiraly-leonard-szigeti-talon-18,kiraly-16,kiraly-szigeti-tanigawa-18,matsuoka-tanigawa-19}), or under cardinality constraints (B\'{e}rczi and Frank \cite{berczi1,berczi2,berczi3,berczi4}).

\begin{thm} [\cite{edmonds}] \label{16}
 	Let $D$ be a digraph and $R_{1}, \ldots, R_{k}$ be nonempty subsets of $V(D)$. There exist arc-disjoint branchings $B_{i}$, $i=1, \ldots, k$, with root sets $R_{i}$ if and only if for any $\emptyset \neq X \subseteq V(D)$,
 	%\[
 	$d^{-}(X) \geq |\{ R_{i}: R_{i} \cap X=\emptyset \}|$.
 	%\]
\end{thm}

We note that Theorem \ref{16} is the base of all our results in this paper.  Edmonds~\cite{edmonds} also studied the problem of packing arc-disjoint spanning arborescences.
 
 \begin{thm}[\cite{edmonds}]\label{26}
 	Digraph $D=(V,A)$ has $k$ arc-disjoint spanning arborescences (possibly rooted at different vertices) if and only if for any disjoint subsets $X_{1}, \ldots, X_{t}$ of $V$,
 	\[
 	\sum_{j=1}^{t}d^{-}(X_{j}) \geq k(t-1).
 	\]
 \end{thm}

Let $F_{1}, \ldots, F_{k}$ be arc-disjoint $x$-arborescences in $D$. 
For $I \subseteq [k]$,  $\overline{I} = [k]\setminus I$.
Denote 
$P_{I}(X)= \{ i \in I : X \cap V(F_{i}) =\emptyset\}$.
%\]
\noindent In particular, we write $P(X)$ for $P_{[k]}(X)$ and $P_{i}(X)$ for $P_{\{ i\}}(X)$.
% \noindent
For $u \in V$, define
\[
w_{I}(u) =
   \begin{cases}
   \min \{ |\{ i \in I : u \notin V(F_{i})\}|, ~  [x,u]_{A \setminus \cup_{i=1}^{k}F_{i}}\},& \text{if $u \in N^{+}(x)$}, \\
   0,& \text{if $u \in V \setminus N^{+}(x)$}.
   \end{cases}
\]
For a set function $f: \Omega \rightarrow \mathbb{R}$, define $\widetilde{f}: 2^{\Omega} \rightarrow \mathbb{R}$ as $\widetilde{f}(X)=\sum_{x \in X} f(x)$, where $X \subseteq \Omega$.

Frank \cite{szego} remarked that Theorem \ref{16} is equivalent to the following. 

\begin{thm} [\cite{szego}] \label{10}
	Let $D=(V+x, A)$ be a digraph, $F_{1}, \ldots, F_{k}$ be $k$ arc-disjoint $x$-arborescences. They can be completed to $k$ arc-disjoint spanning $x$-arborescences
	% $F^{*}_{1}, \ldots, F^{*}_{k}$
	if and only if for any $\emptyset \neq X \subseteq V$,
\begin{equation}\label{11}
d^{-}_{A \setminus \cup_{i=1}^{k}F_{i}} (X)\geq |P(X)|.
\end{equation}
\end{thm}

The following extension is due to Cai \cite{C-6} and Frank \cite{F-13}.

\begin{thm} [\cite{C-6,F-13}] \label{Cai-21}
Let $f: V \rightarrow \mathbb{N}$ and $g: V \rightarrow \mathbb{N}$ be lower and upper bounds for which $f \le g$. A digraph $D =(V,A)$ includes $k$
disjoint spanning arborescences so that each node $v$ is the root of at least $f(v)$ and at most $g(v)$ of these arborescences if
and only if
\begin{itemize}
	\item [(i)] 
%\noindent $(i)$ 
$\widetilde{f}(V) \le k$; 
	\item [(ii)]
%\noindent $(ii)$ 
for any disjoint nonempty subsets $X_{1}, \ldots, X_{t}$ of $V$,
\begin{equation} \label{C-F-1}
\sum^{t}_{j=1}d^{-} (X_{j}) \geq k(t-1) + \widetilde{f}(V \setminus \cup_{j=1}^{t}X_{j});
\end{equation}
	\item [(iii)]
%\noindent and $(iii)$ 
for every subset $ \emptyset \neq X \subseteq V$,
$\widetilde{g}(X) \ge k -  d^{-} (X) $.
\end{itemize}
\end{thm}
Note that the condition $(i)$ $\widetilde{f}(V) \le k$ can be interpreted as the inequality in (\ref{C-F-1}) written for $t =0$.

Our first main result is the following theorem, which also generalizes  Theorem \ref{16}. It is intended to characterize the situation of arborescence augmentation with their root degrees  bounded below.

\begin{thm} \label{1}
For digraph $D =( V +x, A)$ and integer $k >0$,
 % $k>0$, % be a positive .
let $\{ I_{1}, \ldots, I_{l}\}$ be a partition of $[k]$, $c_{1}, \ldots, c_{l}$ be nonnegative integers. Suppose $F_{1}, \ldots, F_{k}$ are  arc-disjoint $x$-arborescences in $D$,
%Then t
then they can be completed to $k$ arc-disjoint spanning $x$-arborescences $F^{*}_{1}, \ldots, F^{*}_{k}$ such that $\sum_{i \in I_{\alpha}} d^{+}_{F^{*}_{i}} (x) \geq c_{\alpha}$ for $\alpha \in [l]$ if and only if for any disjoint subsets $X_{1}, \ldots, X_{t}$ of $V$ and any subset $I$ that is the union of some of $I_{1}, \ldots, I_{l}$,
\begin{equation} \label{4}
 \sum^{t}_{j=1}d^{-}_{A \setminus \cup_{i=1}^{k}F_{i}} (X_{j})\geq  \sum^{t}_{j=1} |P_{I}(X_{j})| +
\sum_{ I_{\alpha} \subseteq  \overline{I}}(c_{\alpha}-\sum_{i \in I_{\alpha}}d^{+}_{F_{i}}(x))- \widetilde{w}_{\overline{I}}(V \setminus \cup_{j=1}^{t}X_{j}).
\end{equation}
In particular, when $t=0$, (\ref{4}) becomes %implies
\begin{equation} \label{27}
\widetilde{w}_{\overline{I}}(V) \geq \sum_{ I_{\alpha} \subseteq  \overline{I}}(c_{\alpha}-\sum_{i \in I_{\alpha}}d^{+}_{F_{i}}(x)).
\end{equation}
\end{thm}

Note that Theorem~\ref{10}  is a special case of Theorem~\ref{1}. 
To see this, let  $I_{i} = \{i\}$ ($1 \le i \le k$), $l=k$, and $c_{1}=\ldots=c_{k}=0$, then (\ref{11}) implies that for any disjoint subsets $X_{1}, \ldots, X_{t}$ of $V$, $I \subseteq [k]$,
$d^{-}_{A \setminus \cup_{i=1}^{k}F_{i}} (X_{j}) \geq |P(X_{j})| \geq |P_{I}(X_{j})|$, and thus (\ref{4}) holds.

Frank~\cite{frank} obtained the following theorem as a corollary of Theorems~\ref{16} and~\ref{10}.
\begin{thm}[\cite{frank}] \label{12} 
	A digraph $D$ can be decomposed into $k$ branchings if and only if the maximum in-degree of $D$ is at most $k$ and its underlying undirected graph can be decomposed into $k$ forests. 
\end{thm} 

As an  application of Theorem~\ref{1}, we deduce the following corollary  which was available in the book of Schrijver as \cite[Theorem 53.3]{Sch}, and  is a strengthened version of Theorem~\ref{12}.  
Indeed, Corollary~\ref{14} will be deduced from Theorem \ref{13}, which is a characterization for branching covering with their root degrees bounded below.  

\begin{cor}\label{14} 
	Let $D = (V,A)$ be a digraph. If $A$ can be covered by $k$ branchings, then $A$ can be covered by $k$ branchings each of size $\lfloor |A|/k \rfloor$ or
	$\lceil |A|/k\rceil $. 
\end{cor}

Our second main result characterizes arborescence augmentation with their root degrees   bounded above, which in some sense is the dual of Theorem~\ref{1}.

\begin{thm} \label{31}
Let $D =( V +x, A)$ be a digraph, $k>0$ be an integer. Let $\{ I_{1}, \ldots, I_{l}\}$ be a partition of $[k]$ and $c'_{1}, \ldots, c'_{l}$ be nonnegative integers. Let $F_{1}, \ldots, F_{k}$ be arc-disjoint $x$-arborescences in $D$ such that $ \sum_{i \in I_{\alpha}}d^{+}_{F_{i}}(x) \leq c'_{\alpha} $ for $1 \leq \alpha \leq l$. Then they can be completed to $k$ arc-disjoint spanning $x$-arborescences $F^{*}_{1}, \ldots, F^{*}_{k}$  %respectively
such that $\sum_{i \in I_{\alpha}}d^{+}_{F^{*}_{i}}(x) \leq c'_{\alpha}$ for $1 \leq \alpha \leq l$ if and only if the following two conditions hold: 
%\begin{enumerate}
%\item[(i)] %
\begin{itemize}
	\item [(i)] 
%\noindent~(i)
for any nonempty $X \subseteq V$, (\ref{11}) holds. 
\item[(ii)] 
%\noindent~(ii)
for any disjoint $X_{1}, \ldots, X_{t} \subseteq V$ and $I \subseteq [k]$ that is the union   of some of $I_{1}, \ldots, I_{l}$,
\begin{equation} \label{22}
\sum_{j=1}^{t}(|P_{I}(X_{j})|-d_{ A \setminus (\cup_{i=1}^{k}F_{i} \cup E^{+}(x)) }^{-}(X_{j}))\leq \sum_{I_{\alpha} \subseteq I}(c'_{\alpha}-\sum_{i \in I_{\alpha}}d^{+}_{F_{i}}(x)).
\end{equation}
\end{itemize}
\end{thm}

As  an  application of Theorem~\ref{31},  we deduce the following corollary,  which is due to B\'{e}rczi and Frank \cite[Theorem 23]{berczi1}, and is a generalization of Theorem \ref{16}.  
For $z \in \mathbb{Z}$, denote $z^{+}:= \max\{z, 0 \}$.    

\begin{cor} \label{BF}  (\cite{berczi1})
	Let $D$ be a digraph and $c_{1}, \ldots, c_{k}$ be positive integers. Then there exist $k$ arc-disjoint branchings $B_{1}, \ldots, B_{k}$ with $|R(B_{i})|=c_{i}$ in $D$ if and only if for any disjoint subsets $X_{1}, \ldots, X_{t}$ of $V(D)$,
	\[
	\sum_{j=1}^{t}d_D^{-}(X_{j}) \geq \sum_{i=1}^{k}(t -c_{i})^{+}. 
	\]
\end{cor}

Combining our two main results (Theorems \ref{1} and \ref{31}), we have the following result which  characterizes  arborescence augmentation  with their root degrees having both lower  and upper bounds.

\begin{thm} \label{32}
Let $D =( V +x, A)$ be a digraph, $k>0$ be an integer. Let $\{ I_{1}, \ldots, I_{l}\}$ be a partition of $[k]$,  $c_{1}, \ldots, c_{l}, c'_{1}, \ldots, c'_{l}$ be nonnegative integers such that $c_{\alpha} \leq c'_{\alpha}$ for $1 \leq i \leq l$. Let $F_{1}, \ldots, F_{k}$ be arc-disjoint $x$-arborescences in $D$ such that $ \sum_{i \in I_{\alpha}}d^{+}_{F_{i}}(x) \leq c'_{\alpha} $ for $1 \leq \alpha \leq l$. Then they can be completed to $k$ arc-disjoint spanning $x$-arborescences $F^{*}_{1}, \ldots, F^{*}_{k}$ such that $c_{\alpha} \leq \sum_{i \in I_{\alpha}}d^{+}_{F^{*}_{i}}(x) \leq c'_{\alpha}$ for $1 \leq \alpha \leq l$ if and only if for any disjoint $X_{1}, \ldots, X_{t} \subseteq V(D)$ and any $I \subseteq [k]$ that is the union of some of $I_{1}, \ldots, I_{l}$,  
\begin{itemize}
\item[(i)] 
%\noindent~(i)
(\ref{4}) holds;  
\item[(ii)] 
%\noindent~(ii)
(\ref{22}) holds.
\end{itemize}
\end{thm}

% In particular, when $t=0$, (\ref{4}) becomes (\ref{27}). 
As an  application of Theorem~\ref{32}, we shall deduce a result that was first discovered by B\'{e}rczi and Frank (\cite[Theorem 3]{berczi2}). This is explained in Section 5.

This paper is organized as follows. In Section 2, we shall 
study the  arborescence augmentation with their root degrees bounded below,  
prove Theorem~\ref{1}. As an application, we shall give a characterization for branching covering of digraphs with their root degrees bounded below, present Theorem \ref{13}, and deduce Corollary~\ref{14} from it.  
In Section 3, we shall introduce  the set $D(\Omega)$ that consists of all families of disjoint subsets of a finite set $\Omega$, define  a partial order $\leq$ on $D(\Omega)$; and study this partial order by \lq\lq properly intersecting elimination operations\rq\rq (which had been used by B\'{e}rczi and Frank in \cite{berczi1}). This section serves as preparation  for the proof of Theorem~\ref{31}.  
In Section 4, we shall study arborescence augmentation with their root degrees  bounded above, prove Theorem~\ref{31}. 
As an application, 
we give a characterization for the existence of arc-disjoint $c^{-}$-branchings, whose root sets contain given vertices; this will be Corollary \ref{24}; then we shall deduce Corollary \ref{BF} from Corollary \ref{24}. We shall also  prove Theorem \ref{32} by combining
Theorems \ref{1} and \ref{31}. 
The final Section 5 contains some remarks. By using the framework on bipartite graphs and supermodular functions (which is due to Lov\'asz \cite{L-30}), we integrate our work of arborescence augmentation to this well-studied framework, and present some more generalized forms of our work.

\section{Arborescence augmentation with their root degrees bounded below, and branching covering} 

Let $\Omega $ be a finite set. Two subsets $X, Y \subseteq \Omega$ are said to be {\em intersecting} if $X \cap Y \neq \emptyset$ and {\em properly intersecting} if $X \cap Y$, $X \setminus Y$, and $Y \setminus X \neq \emptyset$.
%$ \neq \emptyset$.
An {\em (positively) intersecting submodular function} is a set function $f: 2^{\Omega} \rightarrow \mathbb{R}$,  where $2^{\Omega}$ denotes the power set of $\Omega$, which satisfies the condition: for every {properly} intersecting pair $S, T \subseteq \Omega$ (such that $f(X), f(Y)>0$), we have that $f(S)+f(T) \geq f(S \cup T)+f(S \cap T)$. If $-f$ is (positively) intersecting submodular, then $f$ is said to be {\em (positively) intersecting supermodular}.

Let $D=(V+x,A)$ be a digraph, and $F_{1}, \ldots, F_{k}$ be arc-disjoint  $x$-arborescences in $D$.
Suppose $X, Y \subseteq V+x$, $X \cap Y \neq \emptyset$, and $I \subseteq [k]$. Recall that $d^{-}(X)+d^{-}(Y) \geq d^{-}(X \cup Y)+d^{-}(X \cap Y)$. Also, it is easy to check that $|P_{I}(X)|+|P_{I}(Y)| \leq |P_{I}(X \cup Y)|+|P_{I}(X \cap Y)|$. Hence, both the functions $d^{-}$ and  $d^{-}-|P_{I}|$ are intersecting submodular.

\begin{lem}\label{8}
	Suppose $f:2^{\Omega} \rightarrow \mathbb{R}$ is intersecting submodular and nonnegative. If $f(X)=f(Y)=0$ and $X\cap Y \neq \emptyset$, 
	then $f(X\cup Y)=f(X \cap Y)=0$.
\end{lem}

\begin{proof}
	If $X \subseteq Y$ or $Y \subseteq X$, then the lemma clearly holds.
	% Then s
	Suppose $X$ and $Y$ are {properly} intersecting.
	Since $0= f(X)+f(Y) \geq f(X\cup Y)+f(X \cap Y) \geq 0$, we have $f(X\cup Y)=f(X \cap Y)=0$. 
\end{proof}

\subsection{Proof of Theorem~\ref{1}}

\noindent %\textbf{Proof of Theorem~\ref{1}.} 
\textbf{($\Rightarrow$) Necessity:}
For $i \in I$ and $1 \leq j \leq t$,
since $x$-arborescence  $F_{i}$ %in $D$
can be completed to spanning $x$-arborescences $F^{*}_{i}$, we have
$d_{F_{i}^{*} \setminus F_{i}}^{-}(X_{j}) \geq |P_{i}(X_{j})|$. Hence,
\begin{equation}\label{5}
%\begin{split}
\sum \limits^{t}_{j=1} d^{-}_{\cup_{i \in I}F_{i}^{*} \setminus F_{i}}(X_{j})
= \sum \limits^{t}_{j=1} \sum \limits_{i \in I}d^{-}_{F_{i}^{*} \setminus F_{i}}(X_{j})
\geq \sum \limits^{t}_{j=1} \sum \limits_{i \in I}|P_{i}(X_{j})|
=\sum \limits^{t}_{j=1} |P_{I}(X_{j})|.
%\end{split}
\end{equation}

For $u \in V \setminus \cup_{j=1}^{t}X_{j}$, note that for any $i_{0} \in \overline{I}$, $[x,u]_{F^{*}_{i_{0}} \setminus F_{i_{0}}} \leq 1 $, and the equality holds only if $u \notin V(F_{i_{0}})$. Thus 
$$[x,u]_{\cup_{i\in \overline{I}}F^{*}_{i} \setminus F_{i}} \leq \min \{ |\{ i \in \overline{I} : u \notin V(F_{i})\}|, ~ [x,u]_{A \setminus \cup_{i=1}^{k}F_{i}}\}= w_{\overline{I}}(u).$$   
For $I_{\alpha} \subseteq \overline{I}$, since  $\sum \limits_{i \in I_{\alpha}}d^{+}_{F^{*}_{i}}(x) \geq c_{\alpha}$, $d^{+}_{\cup _{i \in I_{\alpha}}F^{*}_{i} \setminus F_{i}}(x) \geq c_{\alpha}-\sum \limits_{i \in I_{\alpha}}d^{+}_{F_{i}}(x)$.
Therefore,    

\begin{equation} \label{6}
\begin{split} 
 &\sum \limits^{t}_{j=1} d^{-}_{\cup_{i \in \overline{I}}F^{*}_{i} \setminus F_{i}}(X_{j})+ \sum \limits_{u \in V \setminus \cup_{j=1}^{t}X_{j}} w_{\overline{I}}(u)  
\geq \sum \limits_{j=1}^{t}[x, X_{j}]_{\cup_{i \in \overline{I}}F^{*}_{i} \setminus F_{i}}+\sum \limits_{u \in V \setminus \cup_{j=1}^{t}X_{j}}[x,u]_{\cup_{i \in \overline{I}}F^{*}_{i} \setminus F_{i}}\\
& = [x,V]_{\cup_{i \in \overline{I}}F^{*}_{i} \setminus F_{i}}
= \sum \limits_{i \in \overline{I}}d^{+}_{F^{*}_{i} \setminus F_{i}}(x)
= \sum \limits_{I_{\alpha} \subseteq \overline{I}} \sum \limits _{i \in I_{\alpha} }d^{+}_{F^{*}_{i} \setminus F_{i}}(x) 
\geq \sum \limits_{ I_{\alpha} \subseteq  \overline{I}}(c_{\alpha}-\sum_{i \in I_{\alpha}}d^{+}_{F_{i}}(x)).  
\end{split}
\end{equation}

Combining (\ref{5}) and~(\ref{6}), we have
\[
\begin{split}
\sum \limits ^{t}_{j=1}d^{-}_{A \setminus \cup_{i=1}^{k}F_{i}} (X_{j})
& \geq \sum \limits ^{t}_{j=1}d^{-}_{\cup_{i=1}^{k}F^{*}_{i} \setminus F_{i}} (X_{j})
= \sum \limits_{j=1}^{t}d^{-}_{\cup_{i \in I}F^{*}_{i} \setminus F_{i}} (X_{j})+\sum \limits_{j=1}^{t}d^{-}_{\cup_{i \in \overline{I}}F^{*}_{i} \setminus F_{i}} (X_{j}) \\
& \geq
\sum \limits ^{t}_{j=1} |P_{I}(X_{j})| +
\sum \limits _{ I_{\alpha} \subseteq  \overline{I}}(c_{\alpha}-\sum_{i \in I_{\alpha}}d^{+}_{F_{i}}(x))- \sum_{u \in V \setminus
	\cup_{j=1}^{t}X_{j}} w_{\overline{I}}(u).
\end{split}
\]

\textbf{($\Leftarrow$) Sufficiency:}
The proof is by induction on the number $\tau$ of $\alpha \in [l]$ such that $ \sum _{i \in I_{\alpha}}d_{F_{i}}^{+}(x) < c_{\alpha}$. If $\tau=0$, then by setting $t=1$ and $I=[k]$ in (\ref{4}), we deduce from Theorem~\ref{10} that  $F_{1}, \ldots, F_{k}$ % $F_{i}$
can be completed to be spanning.
For the induction step, suppose $\tau \geq 1$. For an arc $e \in A$, denote the head of $e$ by $h(e)$.

\begin{claim} \label{7}
	For any fixed $\alpha_{0} \in [l]$ with $ \sum _{i \in I_{\alpha_{0}}}d_{F_{i}}^{+}(x) < c_{\alpha_{0}}$, there exists $e_{0} \in E^{+}_{A \setminus \cup_{i=1}^{k}F_{i}}(x)$ and $i_{0} \in I_{\alpha_{0}}$ such that $h(e_{0}) \notin V(F_{i_{0}})$  and (\ref{11}) still holds after we update $F_{i_{0}}:=F_{i_{0}} + e_{0}$.
\end{claim}

\begin{proof} %\proof
	Suppose, to the contrary, %the contrary,
	that for any $e_{0} \in E^{+}_{A \setminus \cup_{i=1}^{k}F_{i}}(x)$ and $i_{0} \in I_{\alpha_{0}}$ such that $h(e_{0}) \notin V(F_{i_{0}})$, (\ref{11}) does not hold after we update $F_{i_{0}}:=F_{i_{0}} + e_{0}$. 
	By applying (\ref{27}) for $I= \overline{I_{\alpha_{0}}}$, 	 we have
	\[
	\widetilde{w}_{I_{\alpha_{0}}}(V) \geq c_{\alpha_{0}}-\sum_{i \in {I_{\alpha_{0}}}} d^{+}_{F_{i}}(x) > 0.
	\]
	
	Let $u \in V$ such that $w_{I_{\alpha_{0}}}(u)>0$. 
	Since
	$$w_{I_{\alpha_{0}}}(u) = \min \{ | \{i \in I_{\alpha_{0}}: u \notin V(F_{i}) \}|, [x,u]_{D-\cup_{i=1}^{k}F_{i}} \} >0,$$
	there exists $e_{0} \in E^{+}_{A \setminus \cup_{i=1}^{k}F_{i}}(x)$ with $h(e_{0})=u$ and $\{i \in I_{\alpha_{0}}: u \notin V(F_{i}) \} \neq \emptyset$. For any $i_{0} \in \{i \in I_{\alpha_{0}}: u \notin V(F_{i}) \}$, since (\ref{11}) does not hold after we update  $F_{i_{0}}:=F_{i_{0}} + e_{0}$, there exists $X_{i_{0}} \subseteq V$ such that $u \in X_{i_{0}}$, $d_{A \setminus \cup_{i=1}^{k}F_{i}}^{-}(X_{i_{0}})=|P(X_{i_{0}})|$ and $X_{i_{0}} \cap V(F_{i_{0}}) \neq \emptyset$. Choose a maximal $X_{u} \subseteq V$ subject to $u \in X_{u}$ and $d_{A \setminus \cup_{i=1}^{k}F_{i}}^{-}(X_{u})=|P(X_{u})|$. Since $u \in X_{u} \cap X_{i_{0}}$, by Lemma~\ref{8}, $d_{A \setminus \cup_{i=1}^{k}F_{i}}^{-}(X_{u} \cup X_{i_{0}})=|P(X_{u} \cup X_{i_{0}})|$. By maximality of $X_{u}$, we have $X_{i_{0}} \subseteq X_{u}$; combining this with $X_{i_{0}} \cap V(F_{i_{0}}) \neq \emptyset$, we know that $X_{u} \cap V(F_{i_{0}}) \neq \emptyset$. Hence, for any $i \in I_{\alpha_{0}}$, $X_{u} \cap V(F_{i}) \neq \emptyset $. So $P(X_{u}) = P_{\overline{I_{\alpha_{0}}}}(X_{u})$ and $d_{A \setminus \cup_{i=1}^{k}F_{i}}^{-}(X_{u})=|P_{\overline{I_{\alpha_{0}}}}(X_{u})|$.
	
	Also by Lemma~\ref{8} and the maximality, for any $u, v \in V$ such that $w_{I_{\alpha_{0}}}(u),w_{I_{\alpha_{0}}}(v) >0$, we have either $X_{u}=X_{v}$ or $X_{u} \cap X_{v} = \emptyset$.
	%
	%If we set $I= \overline{I}_{\alpha_{0}}$, (\ref{27}) implies that
	%\[
	%\sum_{u \in V}w_{I_{\alpha_{0}}}(u) \geq c_{\alpha_{0}}- \sum_{i \in I_{\alpha_{0}}}d_{F_{i}}^{+}(x) >0.
	%\]
	%So there exists $u \in V$ such that $w_{I_{\alpha_{0}}}(u)>0$.
	%
	If we regard $\mathcal{F}=\{X_{u} : \hbox{where } w_{I_{\alpha_{0}}}(u) >0 \}$ as disjoint subsets of $V$ and let $I= \overline{I_{\alpha_{0}}}$ in (\ref{27}), then (\ref{4}) implies that
	\begin{equation}\label{28}
	\sum_{X \in \mathcal{F}} d^{-}_{A \setminus \cup_{i=1}^{k}F_{i}} (X)\geq  \sum_{X \in \mathcal{F}} |P_{\overline{I_{\alpha_{0}}}}(X)| +
	c_{\alpha_{0}}- \sum_{i \in I_{\alpha_{0}}} d^{+}_{F_{i}}(x).
	\end{equation}
	However, for any $u \in V $ such that $w_{I_{\alpha_{0}}}(u) >0$, $d_{A \setminus \cup_{i=1}^{k}F_{i}}^{-}(X_{u})=|P_{\overline{I_{\alpha_{0}}}}(X_{u})| $, which implies %that
	\begin{equation}\label{66}
	\sum_{X \in \mathcal{F}} d^{-}_{A \setminus \cup_{i=1}^{k}F_{i}} (X) =  \sum_{X \in \mathcal{F}} |P_{\overline{I_{\alpha_{0}}}}(X)|.
	\end{equation}
	By (\ref{28}) and (\ref{66}), we have $\sum_{i \in I_{\alpha_{0}}} d^{+}_{F_{i}}(x) \geq c_{\alpha_{0}}$,  a contradiction to the assumption.
	% of this claim. %\QEDA
\end{proof}

%Let $X_{1}, \ldots, X_{t}$ be disjoint subsets of $V$ and $I \subseteq [k]$ be a union of some of $I_{1}, \ldots, I_{l}$.

For disjoint subsets $X_{1}, \ldots, X_{t}$  of $V$ and $I \subseteq [k]$ that is the union   of some of $I_{1}, \ldots, I_{l}$, define
\[
\begin{split}
F(X_{1}, \ldots, X_{t}; I)  :=
& \sum  ^{t}_{j=1}d^{-}_{A \setminus \cup_{i=1}^{k}F_{i}} (X_{j})+ \sum_{u \in V \setminus
	\cup_{j=1}^{t}X_{j}} w_{\overline{I}}(u) \\
& -  \sum^{t}_{j=1} |P_{I}(X_{j})| -
\sum _{ I_{\alpha} \subseteq  \overline{I}}(c_{\alpha}-\sum_{i \in I_{\alpha}}d^{+}_{F_{i}}(x)).
\end{split}
\]

The following claim
%We shall
observes the updates of $F(X_{1}, \ldots, X_{t}; I) $
% changes of which we will observe
after we update $F_{i_{0}}:=F_{i_{0}} + e_{0}$ for some $i_{0} \in I_{\alpha_{0}} \subseteq [k]$ and $e_{0} \in E^{+}_{D- \cup_{i=1}^{k}F_{i}}(x)$ such that $h(e_{0}) \notin V(F_{i_{0}})$.

\begin{claim}\label{9}
	If $I_{\alpha_{0}} \subseteq \overline{I}$, then $F(X_{1}, \ldots, X_{t};I)$ does not decrease. If $I_{\alpha_{0}} \subseteq I$, then \\
	$F(X_{1}, \ldots, X_{t};I)$ is decreased by at most $1$. 
\end{claim}

\begin{proof}
	%Obviously,
	Observe that,
	for $1 \leq j \leq t$, $|P_{I}(X_{j})|$ does not increase. If $h(e_{0}) \in \cup_{j=1}^{t}X_{j}$, $\sum  ^{t}_{j=1}d^{-}_{A \setminus \cup_{i=1}^{k}F_{i}} (X_{j})$ is decreased by $1$ and $\sum_{u \in V \setminus
		\cup_{j=1}^{t}X_{j}} w_{\overline{I}}(u)$ stays the same. If $h(e_{0}) \notin \cup_{j=1}^{t}X_{j}$, $\sum  ^{t}_{j=1}d^{-}_{A \setminus \cup_{i=1}^{k}F_{i}} (X_{j})$ stays the same and $\sum_{u \in V \setminus
		\cup_{j=1}^{t}X_{j}} w_{\overline{I}}(u)$ is decreased by at most $1$. Hence, $\sum  ^{t}_{j=1}d^{-}_{A \setminus \cup_{i=1}^{k}F_{i}} (X_{j})+ \sum_{u \in V \setminus
		\cup_{j=1}^{t}X_{j}} w_{\overline{I}}(u)$ is decreased by at most $1$. If $I_{\alpha_{0}} \subseteq \overline{I}$, then $\sum _{ I_{\alpha} \subseteq  \overline{I}}(c_{\alpha}-\sum_{i \in I_{\alpha}}d^{+}_{F_{i}}(x))$ is decreased by $1$. If $I_{\alpha_{0}} \subseteq I$, then $\sum _{ I_{\alpha} \subseteq  \overline{I}}(c_{\alpha}-\sum_{i \in I_{\alpha}}d^{+}_{F_{i}}(x))$ stays the same. 
	This proves the claim. 
\end{proof}

%In the final part,
Toward the proof of the sufficiency of Theorem~\ref{1},
we try to find an $\alpha_{0} \in [l]$, $i_{0} \in I_{\alpha_{0}}$, $e_{0} \in E^{+}_{D-\cup_{i=1}^{k}F_{i}}(x)$ such that: $(i)$ $h(e_{0}) \notin V(F_{i_{0}})$; $(ii)$ $ \sum _{i \in I_{\alpha_{0}}} d_{F_{i}}^{+}(x) < c_{\alpha_{0}}$; and $(iii)$ (\ref{4}) still holds after we update $F_{i_{0}}:=F_{i_{0}} + e_{0}$. Add $e_{0}$ to $F_{i_{0}}$ and continue the process until for any $\alpha \in [l]$, $ \sum_{i \in  I_{\alpha}}d_{F_{i}}^{+}(x) \geq c_{\alpha}$; meanwhile, (\ref{4}) sill holds. Since (\ref{4}) implies (\ref{11}) by setting $I=[k]$ and $t=1$,
by Theorem~\ref{10},  $F_{i}$ can be completed to be spanning.

If $\tau=1$, then there exists exactly one $\alpha_{0} \in [l]$ such that $ \sum _{i \in I_{\alpha_{0}}}d_{F_{i}}^{+}(x) < c_{\alpha_{0}}$.
% Applying Claim~\ref{7},  p
Pick  $e_{0} \in E^{+}_{A \setminus \cup_{i=1}^{k}F_{i}}(x)$ and $i_{0} \in I_{\alpha_{0}}$  that is provided by Claim~\ref{7}. 
% corresponding to what Claim~\ref{7} says.
After we update  $F_{i_{0}}:=F_{i_{0}} + e_{0}$, for $1 \leq j \leq t$,
\begin{equation}\label{3}
d^{-}_{A \setminus \cup_{i=1}^{k}F_{i}} (X_{j})\geq |P(X_{j})| \geq |P_{I}(X_{j})|.
\end{equation}
If $I_{\alpha_{0}} \subseteq I$, since $c_{\alpha}-\sum_{i \subseteq I_{\alpha}}d_{F_{i}}^{+}(x) \leq 0 $ for $I_{\alpha } \subseteq \overline{I}$, (\ref{3}) implies (\ref{4}) still holds. If $I_{\alpha_{0}} \subseteq \overline{I}$, then thanks to Claim~\ref{9}, (\ref{4}) still holds.

Suppose $\tau\geq 2$, and without loss of generality, suppose $\sum_{i \in I_{\alpha}}d^{+}_{F_{i}}(x) < c_{\alpha}$ for $\alpha = 1, 2$. Since (\ref{4}) holds for any disjoint subsets $X_{1}, \ldots, X_{t}$ of $V$ and any subsets $I$ as the union of some of $I_{1} \cup I_{2},I_{3},\ldots,I_{l}$, by induction hypothesis, $F_{i}$ can be completed to $k$ arc-disjoint spanning  $x$-arborescences $F'_{i}$ such that $\sum_{i \in I_{1} \cup I_{2}}d^{+}_{F'_{i}}(x) \geq c_{1}+c_{2}$ and $\sum_{i \in I_{\alpha} }d^{+}_{F'_{i}}(x) \geq c_{\alpha}$ for $3 \leq \alpha \leq l$. Since $\sum_{i \in I_{1} \cup I_{2}}d^{+}_{F'_{i}}(x) \geq c_{1}+c_{2}$, we have either  $\sum_{i \in I_{1} }d^{+}_{F'_{i}}(x) \geq c_{1}$ or $\sum_{i \in I_{2} }d^{+}_{F'_{i}}(x) \geq c_{2}$, and without loss of generality, suppose the former inequality holds.

\noindent \textbf{Case 1.} $\sum_{i \in I_{2} }d^{+}_{F'_{i}}(x) > \sum_{i \in I_{2} }d^{+}_{F_{i}}(x)$.

In this case, there exists $e_{0} \in E_{F'_{i_{0}} \setminus F_{i_{0}}}^{+}(x)$ for some $i_{0} \in I_{2}$. Update $F_{i_{0}}:=F_{i_{0}}+e_{0}$. Let $I$ be a union of some of $I_{1}, \ldots, I_{l}$.
If $I_{2} \subseteq I$, since $F_{i}$ can be completed to $F'_{i}$ for $i \in I$, (\ref{5}) holds; since for any $I_{\alpha} \subseteq \overline{I}$, $\sum_{i \in I_{\alpha}} d_{F'_{i}}^{+}(x) \geq c_{\alpha}$, (\ref{6}) holds. Thus (\ref{4}) still holds.
If $I_{2} \subseteq \overline{I}$, by Claim~\ref{9}, $F(X_{1}, \ldots, X_{t};I)$ does not decrease, and thus (\ref{4}) still holds.

\noindent \textbf{Case 2.} $\sum_{i \in I_{2} }d^{+}_{F'_{i}}(x) = \sum_{i \in I_{2} }d^{+}_{F_{i}}(x)$.

Recall that our proof is by induction on the number $\tau$ of $\alpha \in [l]$ such that $ \sum _{i \in I_{\alpha}}d_{F_{i}}^{+}(x)$ $<$ $c_{\alpha}$.
By setting $ c_{1}:=\sum_{i \in I_{1}}d^{+}_{F_{i}}(x) $, the number $\tau$ is reduced by $1$; by the induction hypothesis,
%(and by setting $ c_{1}:=\sum_{i \in I_{1}}d^{+}_{F_{i}}(x) $),
$F_{1}, \ldots, F_{k}$ can be completed to arc-disjoint spanning $x$-arborescences $F''_{1}, \ldots, F''_{k}$ %respectively
such that $\sum_{i \in I_{1} }d^{+}_{F''_{i}}(x) \geq \sum_{i \in I_{1} }d^{+}_{F_{i}}(x) $  and $\sum_{i \in I_{\alpha} }d^{+}_{F''_{i}}(x) \geq c_{\alpha}$ for $2 \leq \alpha \leq l$.
Pick an $e_{0} \in E_{F''_{i_{0}} \setminus F_{i_{0}}}^{+}(x)$ for some $i_{0} \in I_{2}$ and update $F_{i_{0}}:=F_{i_{0}}+e_{0}$. 

Let $I$ be a union of some of $I_{1}, \ldots, I_{l}$.
If $ I_{1} \subseteq I$, since $F_{i}$ can be completed to $F''_{i}$ for $ i \in I$, (\ref{5}) holds; since $\sum_{i \in \overline{I}}d^{+}_{F''_{i}}(x) \geq \sum_{I_{\alpha} \subseteq \overline{I}}c_{\alpha}$, (\ref{6}) holds. Thus (\ref{4}) still holds.
If $I_{2} \subseteq \overline{I}$, by Claim~\ref{9}, $F(X_{1}, \ldots, X_{t};I)$ does not decrease, and thus (\ref{4}) holds. 

The only left case is that
$ I_{1} \subseteq \overline{I}$ and $I_{2} \subseteq I$.
Since $\sum_{i \in I_{2}}d_{F_{i}}^{+}(x) < c_{2}$ and $\sum_{i \in I_{2} }d^{+}_{F'_{i}}(x)$ $=$ $\sum_{i \in I_{2} }d^{+}_{F_{i}}(x)$ (the assumption of this case), $\sum_{i \in I_{1} }d^{+}_{F'_{i}}(x) \geq c_{1} + c_{2}-\sum_{i \in I_{2} }d^{+}_{F'_{i}}(x)$ $\geq$ $c_{1}+1$.
Since $F_{i}$ can be completed to $F'_{i}$ for $i \in I$, (\ref{5}) holds; since $\sum_{i \in \overline{I}}d^{+}_{F'_{i}}(x)= \sum_{i \in I_{1}}d^{+}_{F'_{i}}(x)+ \sum_{I_{\alpha} \subseteq \overline{I \cup I_{1}}}\sum_{i \in I_{\alpha}}d^{+}_{F'_{i}}(x)> \sum_{I_{\alpha} \subseteq \overline{I}} c_{\alpha} $, (\ref{6}) holds, but the equality of (\ref{6}) does not hold.
This proves $F(X_{1}, \ldots, X_{t};I)$ $>$ $0$.
By Claim~\ref{9}, $F(X_{1}, \ldots, X_{t};I)$ is decreased by at most $1$ when we update  $F_{i_{0}}:=F_{i_{0}}+e_{0}$, and thus (\ref{4}) holds.

This finishes the proof of Theorem~\ref{1}.
\QEDA

\subsection{Branching covering with their root degrees bounded below}  

As an application of Theorem~\ref{1}, we give a characterization for branching covering with their root degrees bounded below, which is Theorem \ref{13}. From it, we  deduce Corollary \ref{14}, which is a strengthened version of Theorem~\ref{12}.

A {\em decomposition} of a graph $G$ is a set of edge-disjoint subgraphs with
union $G$. The {\em arboricity} of $G$ is the minimum size of a decomposition
of $G$ into forests.  The {\em fractional arboricity} of $G$, introduced by
Payan~\cite{CP} (also \cite{PJ}) and here denoted $\iarb_1(G)$, is defined
by $\iarb_1(G)= \max_{\nul\ne H \subseteq G}  \frac {\E{H}} {\V{H}-1}.$
The Arboricity Theorem of Nash-Williams~\cite{NW} characterizes when a graph
has arboricity at most $k$.

For a digraph $D$, the {\em fractional arboricity} of $D$ is  the fractional arboricity of its underlying graph,  written as $\iarb_1(D)$. 
Let $\Delta^{-}(D)$ be the maximum in-degree of $D$. 

\begin{thm} \label{13}
	A digraph $D$ can be decomposed into $k$ $c^{+}$-branchings if and only if $\Delta^{-}(D) \leq k$,  {$\iarb_{1}(D) \leq k$} and  {$\frac{|E(D)|}{|V(D)|-c} \le k$}.
\end{thm}
\begin{proof}
	The necessity is obviously true. Next, we prove the sufficiency. Using  $\Delta^{-}(D) \leq k$, construct a new digraph $D'$ from $D$ by adding a new vertex $x $ and arcs from $x$ to $V(D) $ such that for any $u \in V(D)$, $d^{-}_{D'}(u)=k$. Then since {$\iarb_{1}(D) \leq k$}, for any nonempty $X \subseteq V(D)$, we have {$|E(X)| \leq k(|X|-1)$ (where $E(X)$ denotes the set of arcs with both ends in $X$)}. It follows that 
	\begin{equation}\label{29}
	{d^{-}_{D'}(X) = \sum_{u \in X}d^{-}_{D'}(u)-|E(X)| \geq k|X|-k(|X|-1)=k.}
	\end{equation}
	
	Since {$\frac{|E(D)|}{|V(D)|-c} \le k$}, we have  {$|E(D)| \leq  k|V(D)|-kc$}. It follows  that 
	\begin{equation}\label{30}
	{[x,V(D)]_{D'} = \sum_{u \in V(D)}d^{-}_{D'}(u)-|E(D)| = k|V(D)| -|E(D)| \geq kc.}
	\end{equation}
	
	For any partition $\{X_{1},\ldots, X_{t+2} \}$ of $V(D)$ and $I \subseteq [k]$, define
	\[
	G(X_{1}, \ldots, X_{t+2};I):=\sum_{j=1}^{t}d^{-}_{D'}(X_{j}) -(t|I| + c|\overline{I}|-[x, X_{t+1}]_{D'}-|\overline{I}||X_{t+2}|).
	\]
	
	In the above notation, if $I = [k]$ (and then $\overline{I} = \emptyset$), we have
	\begin{align*}
	G(X_{1}, \ldots, X_{t+2};[k])
	& =\sum_{j=1}^{t}d^{-}_{D'}(X_{j}) -(tk -[x, X_{t+1}]_{D'})\\
	& \geq \sum_{j=1}^{t}d^{-}_{D'}(X_{j}) -tk \geq 0; &\text{(by (\ref{29}))}
	\end{align*}
	If $I = \emptyset $ (and then $\overline{I} = [k]$), we have
	\begin{align*}
	G(X_{1}, \ldots, X_{t+2};\emptyset)
	= & \sum_{j=1}^{t}d^{-}_{D'}(X_{j})+[x, X_{t+1}]_{D'} + k|X_{t+2}|-ck \\
	\geq & \sum_{j=1}^{t}[x, X_{j}]_{D'} +[x, X_{t+1}]_{D'} + k|X_{t+2}| -ck \\
	\geq & \sum_{j=1}^{t+2}[x, X_{j}]_{D'}-ck \\
	&\text{(for $u \in X_{t+2}$,  since $k=d^{-}_{D'}(u) \geq [x, u]_{D'}$, $k|X_{t+2}|\geq [x,X_{t+2}]_{D'} $) } \\	
	\geq & 0. ~~~~~ (\text{by}~(\ref{30}))\\
	\end{align*}
	Since $G(X_{1}, \ldots, X_{t+2};I)$ is linear on $|I|$, we deduce that $G(X_{1}, \ldots, X_{t+2};I)\geq 0$.
	
	To apply Theorem~\ref{1}, let $F_{1}, \ldots, F_{k}$ be empty subdigraphs of $D'$ with vertex set $\{ x \}$; for $1 \leq \alpha \leq k$, let $I_{\alpha}=\{ \alpha\}$ and $c_{\alpha}=c$. Then for $I\subseteq [k]$, $w_{\overline{I}}(u) = \min \{[x,u]_{D'}, |\overline{I}| \}$.
	For disjoint subsets $X_{1}, \ldots, X_{t}$ of $V(D)$, let $X_{t+1}:=\{ u \in V(D) \setminus \cup_{j=1}^{t}X_{j}: [x,u]_{D'} \leq |\overline{I}| \}$ and $X_{t+2}:= V(D) \setminus \cup_{j=1}^{t+1}X_{j}$. Since $ P_{I}(X_{j})=I$, $ \sum_{I_{\alpha } \subseteq \overline{I}} (c_{\alpha}- \sum_{i \in I_{\alpha}}d_{F_{i}}^{+}(x))=c|\overline{I}| $ and $ \widetilde{w}_{\overline{I}}(V \setminus \cup_{j=1}^{t}X_{j})= [x, X_{t+1}]_{D'}+|\overline{I}||X_{t+2}|$,  $G(X_{1}, \ldots, X_{t+2};I)\geq 0$ implies (\ref{4}) holds.
	
	By Theorem~\ref{1}, $D'$ has $k$ arc-disjoint spanning $x$-arborescences $F^{*}_{1}, \ldots, F^{*}_{k}$ such that $d^{+}_{F^{*}_{i}}(x) \geq c$ for $1 \leq i \leq k$. Since {$|E(D')|=k|V(D)|$, $\{ F^{*}_{1}, \ldots, F^{*}_{k} \} $} is a partition of {$E(D')$}. This proves that $D$ can be decomposed into $k$  $c^{+}$-branchings $F^{*}_{1}-x, \ldots, F^{*}_{k}-x$. %\QEDA
\end{proof}

From Theorem \ref{13}, we deduce the following corollary. Note that  Corollary \ref{14-A} is equivalent with Corollary \ref{14},  which was available in the book of Schrijver as \cite[Theorem 53.3]{Sch}.  

\begin{cor}\label{14-A}
	For digraph $D$ and integer $k>0$, suppose $\Delta^{-}(D) \leq k$,  $\iarb_{1}(D) \leq k$ and
	%$\frac{a(D)}{v(D)-\frac{c}{k}} = k$.
	$c= k  |V(D)| - |E(D)|$.
	Then $D$ can be decomposed into $k$ branchings, each of which is a $\lfloor \frac{c}{k}\rfloor$-branching or $\lceil \frac{c}{k}\rceil $-branching.
\end{cor} 
\begin{proof}  %\proof
	Applying Theorem~\ref{12}, suppose $D$ can be decomposed into $k$ arc-disjoint branchings $F_{1}, \ldots, F_{k}$ with $|R(F_{i})|=c_{i}$. Then $\sum_{i=1}^{k}c_{i}=c$. Next we show that if $|c_{i}-c_{j}| >1$, then $F_{i} \cup F_{j}$ can be decomposed into branchings $F_{i}'$ and $F_{j}'$ such that $||R(F'_{i})| -|R(F'_{j})|| \leq 1$. Let $D':=F_{i} \cup F_{j}$. Then {$2|V(D')|-|E(D')|=c_{i}+c_{j} \geq 2 \lfloor \frac{c_{i}+c_{j}}{2}\rfloor$}. By Theorem~\ref{13}, $D'$ can be decomposed into two $\lfloor \frac{c_{i}+c_{j}}{2}\rfloor^{+}$-branchings $F_{i}'$ and $F_{j}'$. Since $|R(F_{i}')|+|R(F_{j}')|=c_{i}+c_{j}$, $|R(F_{i}')|, |R(F_{i}')| \leq \lceil \frac{c_{i}+c_{j}}{2}\rceil$. Hence $||R(F_{i}')| -|R(F_{j}')|| \leq 1$.

	By recursively adjusting $F_{i}$ and $F_{j}$ with $c_{i}$ maximum and $c_{j}$ minimum using the above procedure,  %as the above claim says,
	we can reduce the set $ \{(i, j): ||R(F_{i})|-|R(F_{j})|| \mbox{ is  maximum for all } {i,j \in [k]} \}$ until we obtain $k$ branchings, each of which is a $\lfloor \frac{c}{k}\rfloor$-branching or $\lceil \frac{c}{k}\rceil $-branching. %\QEDA 
\end{proof}

\section{{Properly} intersecting elimination operations} 

In this section, we shall study some operations that had been used by B\'{e}rczi and Frank in \cite{berczi1}. For convenience of the reader, we have made this section self-contained. This section serves as preparation  for the proof of Theorem~\ref{31}.   

Let $\Omega$ be a finite set. Let {\em $D(\Omega)$ be the set that consists of all families of disjoint subsets of $\Omega$}. We define {\em a partial order $\leq$ on $D(\Omega)$}. Suppose $\mathcal{F}_{1}, \mathcal{F}_{2} \in D(\Omega)$. We say $\mathcal{F}_{1}$ is a lower bound of $\mathcal{F}_{2}$ (or equivalently, $\mathcal{F}_{2}$ is an upper bound of $\mathcal{F}_{1}$), written as $\mathcal{F}_{1} \leq \mathcal{F}_{2} $, if for any $X \in \mathcal{F}_{1}$, there exists a $Y \in \mathcal{F}_{2}$ such that $X \subseteq Y$. Denote by $\mathcal{F}_{1} \vee \mathcal{F}_{2}$ and $\mathcal{F}_{1} \wedge \mathcal{F}_{2}$ the least common upper bound and the greatest common lower bound of $\mathcal{F}_{1}$ and $\mathcal{F}_{2}$ respectively.

Let $\mathcal{F}$ be a multiset, which consists of some subsets of $\Omega$ (these subsets do not have to be different). Let $\cup \mathcal{F}$ be the union of elements in $\mathcal{F}$ (then $\cup \mathcal{F} \subseteq \Omega$).
%??, is this correct? DY}.
Let $x \in \Omega $ and $\mathcal{F}(x) $ denote the number of elements  in $\mathcal{F}$ containing $x$. If there exist no {properly} intersecting pairs in $\mathcal{F}$, then we call $\mathcal{F}$ {\em laminar}. If there exists a {properly} intersecting pair $X$ and $Y$ in $\mathcal{F}$, then we call it a {\em {properly} intersecting elimination operation (PIEO for simplicity) on $X $ and $ Y$ in $\mathcal{F}$} if we obtain $\mathcal{F}'$ by replacing $X$ and $Y $ with one of the following three types of subset(s):

%\begin{enumerate} 
%\item[(I)] %[type~1] %
\noindent {\em Type $1$},
$X \cup Y$ and $X \cap Y$, denoted as $\mathcal{F} \xrightarrow{1} \mathcal{F}'$;

%\item[(II)] [(type $2$)]
\noindent {\em Type $2$},  $X \cup Y$, denoted as $\mathcal{F} \xrightarrow{2} \mathcal{F}'$;

%\item[(III)] [type 3]
\noindent {\em Type $3$}, $X \cap Y$, denoted as $\mathcal{F} \xrightarrow{3} \mathcal{F}'$.
%\end{enumerate}

\smallskip 

Let $Z_{1}$ and $Z_{2}$ be multisets. Denote by $Z_{1} \uplus Z_{2}$ the multiset union of $Z_{1}$ and $Z_{2}$, that is, for any $z$, the number of $z$ in $Z_{1} \uplus Z_{2}$ is the total number of $z$ in $Z_{1}$ and $Z_{2}$.

From now on till the end of this section, we suppose $\mathcal{F}_{1}, \mathcal{F}_{2} \in D(\Omega)$. %Suppose w
We adopt PIEOs in $\mathcal{G}_{0}=\mathcal{F}_{1} \uplus \mathcal{F}_{2}$, step by step, and obtain families   $\mathcal{G}_{0}, \ldots, \mathcal{G}_{i-1},\mathcal{G}_{i},\ldots$ of subsets  of $\Omega$.
Let $\mathcal{G}_{i}'$ be the family of maximal elements in $\mathcal{G}_{i}$.

\begin{proposition} \label{clm-i-n-PIE-P}
	For any  $v \in \Omega$ and $ i \ge 1$ in the above process,  $\mathcal{G}_{i-1}(v) \geq \mathcal{G}_{i}(v)$.
\end{proposition}

\begin{proof}
	Suppose we adopt the PIEO on $X$ and $Y$ in $\mathcal{G}_{i-1}$.
	Then $\mathcal{G}_{i-1}-\{X, Y\} \supseteq \mathcal{G}_{i}-\{X\cup Y, X \cap Y\}$.
	For $v \in \Omega$, since $\{ X, Y\}(v)=\{ X \cup Y, X \cap Y\}(v)$, we have
	\begin{equation}\label{46}
	\begin{split}
	\mathcal{G}_{i-1}(v)
	& =(\mathcal{G}_{i-1}-\{X, Y \})(v) + \{X, Y \} (v)  \\
	& \geq (\mathcal{G}_{i}-\{X \cup Y, X \cap Y \})(v) + \{X\cup Y, X \cap Y \} (v) \\
	& \geq \mathcal{G}_{i}(v).
	\end{split}
	\end{equation}
\end{proof}
% denoted by PIE($\mathcal{F}_{1}, \mathcal{F}_{2}$).

\begin{proposition}\label{47-P}
	If $X, Y  \in \mathcal{G}_{i} $ are {properly} intersecting, then $X, Y \in \mathcal{G}_{i}'$.
\end{proposition}

\begin{proof}
	Suppose there exists $Z \in \mathcal{G}_{i} $ such that $X \subsetneq Z $. For $v \in X \cap Y$, $ \mathcal{G}_{i}(v) \geq \{X, Y,Z \}(v) \geq 3 $; but by Proposition~\ref{clm-i-n-PIE-P}, $\mathcal{G}_{i}(v) \leq \mathcal{G}_{0}(v)= \mathcal{F}_{1}(v) + \mathcal{F}_{2}(v) \leq  2$; a contradiction. So $X $ is maximal in $\mathcal{G}_{i}$, and the same for $Y$.
\end{proof}

Note once we adopt the PIEO on a {properly} intersecting pair in  $\mathcal{G}_{i-1}$, if $\mathcal{G}_{i-1} \xrightarrow{2~or~3} \mathcal{G}_{i}$, then $|\mathcal{G}_{i-1}| > |\mathcal{G}_{i}|$. If $\mathcal{G}_{i-1} \xrightarrow{1} \mathcal{G}_{i}$, by Proposition~\ref{47-P}, the number of maximal elements in $\mathcal{G}_{{i}}$ is less than that in $
\mathcal{G}_{{i-1}}$. Thus the process of PIEOs will terminate.
Suppose the obtained  families of subsets  of $\Omega$ are
%$\mathcal{F}_{1} \uplus \mathcal{F}_{2}=$
$\mathcal{G}_{0}, \ldots, \mathcal{G}_{n}$. Then  $\mathcal{G}_{n}$ is laminar.

\begin{proposition} \label{52-P}
	Let $i_{0} \in [n]$.	Suppose for $i \in [i_{0}]$, 
	$\mathcal{G}_{i-1} \xrightarrow{1~or~2} \mathcal{G}_{i}$. Then for $i \in [i_{0}]$ and $Z \in \mathcal{G}_{i}'$, $Z$ contains an element in $\mathcal{G}_{0}$. In particular, if $Z \notin \mathcal{G}_{0}$, then $Z$ contains an element in $\mathcal{F}_{j}$ for $j=1,2$.
\end{proposition}

\begin{proof}
	We prove the proposition by induction on $i \in [i_{0}]$
	%(we leave it to readers),
	and we only need to show the induction step. Suppose we replace a {properly} intersecting pair $X$ and $Y$ in $\mathcal{G}_{i-1}$ with $ X \cup Y$ and possibly $X \cap Y$ and obtain $\mathcal{G}_{i}$. By Proposition~\ref{47-P}, $X, Y \in \mathcal{G}_{i-1}'$ and thus $X \cup Y \in \mathcal{G}_{i}'$. So $\mathcal{G}_{i}'$ consists of $X \cup Y$ and all the subsets in $\mathcal{G}_{i-1}'$ not contained in $X \cup Y$. Note that, if $X, Y \in \mathcal{G}_{0}$, since $X$ and $Y$ are {properly} intersecting, $X$ and $Y$ do not belong to the same $\mathcal{F}_{j}$ for $j=1,2$; and thus $X \cup Y$ contains an element in $\mathcal{F}_{j}$ for each $j=1,2$. And applying the induction hypothesis, we prove the induction step.
\end{proof}

Define $\mathcal{F}_{3}:= \mathcal{G}_{n}' $ and $\mathcal{F}_{4}:=\mathcal{G}_{n} \setminus \mathcal{F}_{3} $. Obviously, we have $\mathcal{F}_{4} \le \mathcal{F}_{3}$.  

\begin{proposition}\label{17-P} The following hold true: 
\begin{enumerate}
	\item[(i)] 	$\mathcal{F}_{3}, \mathcal{F}_{4} \in D(\Omega)$. $\cup\mathcal{F}_{4} \subseteq (\cup \mathcal{F}_{1}) \cap (\cup \mathcal{F}_{2})$. 
	
	\item[(ii)]
	 Moreover, $\cup\mathcal{F}_{4} = (\cup \mathcal{F}_{1}) \cap (\cup \mathcal{F}_{2})$ if and only if for any $i \in [n]$, $\mathcal{G}_{i-1} \xrightarrow{1} \mathcal{G}_{i}$.
\end{enumerate}
\end{proposition}

\begin{proof} Since $\mathcal{G}_{n}$ is laminar, 
	we know that $\mathcal{F}_{3} \in D(\Omega)$  and $\cup \mathcal{F}_{3}= \cup \mathcal{G}_{n}$. 
	
	Let $u \in \cup \mathcal{F}_{4}$. Since $ \cup \mathcal{F}_{4} \subseteq \cup \mathcal{G}_{n} = \cup \mathcal{F}_{3}$, we know that $\mathcal{F}_{3}(u), \mathcal{F}_{4}(u) \geq 1$. By Proposition  \ref{clm-i-n-PIE-P} and $\mathcal{F}_{1}, \mathcal{F}_{2} \in D(\Omega)$,  we have
	\[
	2 \le  \mathcal{F}_{3}(u) + \mathcal{F}_{4}(u)  = \mathcal{G}_{n}(u)  \leq \mathcal{G}_{0}(u) = \mathcal{F}_{1}(u) + \mathcal{F}_{2}(u) \leq 2.
	\]
	Therefore, $ \mathcal{F}_{4}(u)=1$ and $\mathcal{F}_{1}(u)=\mathcal{F}_{2}(u)=1$.
	This proves that  $\mathcal{F}_{4} \in D(\Omega)$; 
	and $u \in (\cup \mathcal{F}_{1}) \cap (\cup \mathcal{F}_{2}) $,  
	hence $ \cup \mathcal{F}_{4} \subseteq (\cup \mathcal{F}_{1}) \cap (\cup  \mathcal{F}_{2})$.

	Suppose for any $i \in [n] $, $\mathcal{G}_{i-1} \xrightarrow{1} \mathcal{G}_{i}$. Let $v \in \Omega$. Then the equality of (\ref{46}) holds.  Thus $\mathcal{G}_{0}(v)=2$, that is $ v \in (\cup \mathcal{F}_{1}) \cap (\cup \mathcal{F}_{2})$, implies $\mathcal{G}_{n}(v)=2$, that is $v \in \cup \mathcal{F}_{4}$. Hence, $ (\cup \mathcal{F}_{1}) \cap (\cup \mathcal{F}_{2}) \subseteq \cup \mathcal{F}_{4} $. Conversely, suppose $ \cup \mathcal{F}_{4}= (\cup \mathcal{F}_{1}) \cap (\cup \mathcal{F}_{2})$. And suppose for some $i_{0} \in [n]$, we adopted the PIEO of Type 2 or 3 on $X$ and $Y$ in $\mathcal{G}_{i_{0}-1}$ and obtained $\mathcal{G}_{i_{0}}$. Let $w \in X\cap Y$. Then $2=\mathcal{G}_{i_{0}-1}(w) > \mathcal{G}_{i_{0}}(w)=1$. Since $\mathcal{G}_{0}(w) \geq \mathcal{G}_{i_{0}-1}(w) > \mathcal{G}_{i_{0}}(w) \geq \mathcal{G}_{n}(w) $, we have $\mathcal{G}_{0}(w)=2$, that is $w \in (\cup \mathcal{F}_{1}) \cap (\cup \mathcal{F}_{2})$, and $ \mathcal{G}_{n}(w) \leq 1$, that is $w \notin \cup \mathcal{F}_{4}$, a contradiction.  
\end{proof}

\begin{proposition} \label{18-P}
	Suppose for $i \in [n]$, $\mathcal{G}_{i-1} \xrightarrow{1~or~2} \mathcal{G}_{i}$, then the following assertions hold.
	\begin{itemize}
		\item[(i)] $\mathcal{F}_{3}=\mathcal{F}_{1} \vee \mathcal{F}_{2}$.
		
		\item[(ii)] If $\cup \mathcal{F}_{2} \subseteq \cup \mathcal{F}_{1}$, then $|\mathcal{F}_{3}| \leq |\mathcal{F}_{1}|$. The equality holds if and only if $\mathcal{F}_{2} \leq \mathcal{F}_{1} $.
	\end{itemize}
\end{proposition}

%\proof

\begin{proof}  %First, we  %prove the following claims:
	First, we show the following claims.
	\begin{itemize}
		\item[(a)]   %(a)
		For any $Z_{1} \in \mathcal{F}_{1} \cup \mathcal{F}_{2}$, there exists $Z_{2} \in \mathcal{G}'_{i}$ such that $Z_{1} \subseteq Z_{2}$.
			
		% (b)
		\item[(b)] For any $Z_{3} \in \mathcal{G}'_{i}$, there exists $Z_{4} \in \mathcal{F}_{1} \vee \mathcal{F}_{2}$ such that $Z_{3} \subseteq Z_{4}$.
	\end{itemize}
	
	The proof of these two claims is by induction on $i$.
	%Suppose $Z_{1} \in \mathcal{F}_{1} \cup \mathcal{F}_{2}$, and $Z_{3} \in \mathcal{G}'_{i}$.
	For the base step, $(a)$  and $(b)$ hold for $\mathcal{G}'_{0}$.
	% and $(b)$ holds for  $\mathcal{G}'_{0}$.
	For induction hypothesis, suppose $(a)$  and $(b)$ hold for $\mathcal{G}'_{i-1}$. Suppose we adopt the PIEO  of Type $1$ or Type $2$ on $X$ and $Y$ in $\mathcal{G}_{i-1}$.
	
	For $(a)$, suppose $Z_{1} \in \mathcal{F}_{1} \cup \mathcal{F}_{2}$,  by induction hypothesis,
	there exists $Z_{2} \in \mathcal{G}'_{i-1}$ such that $Z_{1} \subseteq Z_{2}$.
	%Since the PIEO is of Type $1$ or Type $2$, %we have
	Then either $Z_{2} \in \mathcal{G}'_{i}$ or $Z_{2} \subseteq X \cup Y$. Thus $(a)$ holds for $\mathcal{G}'_{i}$.

	For $(b)$, suppose $Z_{3} \in \mathcal{G}'_{i}$. 
	Then either $Z_{3} \in \mathcal{G}'_{i-1}$ or $Z_{3}=X \cup Y$. If $Z_{3} \in \mathcal{G}'_{i-1}$, by induction hypothesis, there exists $Z_{4} \in \mathcal{F}_{1} \vee  \mathcal{F}_{2}$ such that $Z_{3} \subseteq Z_{4} $. If $Z_{3} = X \cup Y$, by induction hypothesis, there exist $Z_{5}, Z_{6} \in \mathcal{F}_{1} \vee \mathcal{F}_{2}$ such that $X \subseteq Z_{5}$ and $Y \subseteq Z_{6}$. Since $X \cap Y \neq \emptyset$, thus $Z_{5} \cap Z_{6} \neq \emptyset$; so $Z_{5}=Z_{6}$. Then $Z_{3} \subseteq Z_{5}$. Hence, $(b)$ holds.
	
	Now  $(a)$ implies $ \mathcal{F}_{1}, \mathcal{F}_{2} \leq \mathcal{G}'_{n} =\mathcal{F}_{3}$. So $ \mathcal{F}_{1} \vee \mathcal{F}_{2} \leq \mathcal{F}_{3}$. And $(b)$ implies $\mathcal{F}_{3} \leq \mathcal{F}_{1} \vee \mathcal{F}_{2}$. Hence, $(i)$ holds.
	
	Suppose $\cup \mathcal{F}_{2} \subseteq \cup \mathcal{F}_{1} $. Note that for $v \in \Omega$ and $0 \leq i \leq n$, $v \in \cup \mathcal{G}_{i}$ if and only if $\mathcal{G}_{i}(v) \geq 1$. Since $\mathcal{G}_{0}(v) \geq \mathcal{G}_{n}(v)$, $ \cup \mathcal{G}_{0} \supseteq \cup \mathcal{G}_{n}$;  since  $\cup \mathcal{F}_{2} \subseteq \cup \mathcal{F}_{1} $, $ \cup \mathcal{F}_{1} =  \cup \mathcal{G}_{0} \supseteq \cup \mathcal{G}_{n} = \cup \mathcal{F}_{3}$.
	By $(i)$, $\cup \mathcal{F}_{1} \subseteq \cup \mathcal{F}_{3}$. This proves  $\cup \mathcal{F}_{1} = \cup \mathcal{F}_{3}$.
	
	Also by $(i)$, for any $Y \in \mathcal{F}_{1}$, there exists $X \in \mathcal{F}_{3}$ such that $Y \subseteq X$; since $\cup \mathcal{F}_{1} = \cup \mathcal{F}_{3}$, for any $X \in \mathcal{F}_{3}$, $X = \cup \{Y \in \mathcal{F}_{1}: Y \subseteq X \}$. Hence, $|\mathcal{F}_{3}| \leq |\mathcal{F}_{1}|$. If $\mathcal{F}_{2} \leq \mathcal{F}_{1}$, the equality clearly holds. Conversely, suppose the equality holds. Then for any $X \in \mathcal{F}_{3}$, $ |\{Y \in \mathcal{F}_{1}: Y \subseteq X \}|=1$, that is, there exists $Y \in \mathcal{F}_{1}$ such that $X=Y$. Hence, $\mathcal{F}_{1} = \mathcal{F}_{3}=\mathcal{F}_{1} \vee \mathcal{F}_{2}$, that is $\mathcal{F}_{2} \leq \mathcal{F}_{1}$.
\end{proof}

\section{Arborescence augmentation with their root degrees bounded above} 

\subsection{Proof of Theorem~\ref{31}}

\noindent \textbf{($\Rightarrow$) Necessity:}
By Theorem~\ref{10}, (\ref{11}) holds.  This proves $(i)$.

To prove $(ii)$, note that for $i \in [k]$, since $F^{*}_{i}$ is a spanning $x$-arborescence of $D$, we have
for $X \subseteq V$, if $ X \cap (V(F_{i}) \cup N^{+}_{F^{*}_{i}}(x)) =\emptyset$, then $d_{F_{i}^{*} \setminus (F_{i} \cup E^{+}(x))}^{-}(X) \geq 1$.
For any disjoint $X_{1}, \ldots, X_{t} \subseteq V$ and nonempty $I \subseteq [k]$ that is the union of some of $I_{1}, \ldots, I_{l}$,  
\begin{equation} \label{19}
\begin{split}
\sum_{j=1}^{t} & |P_{I}(X_{j})|- \sum_{I_{\alpha} \subseteq I}(c'_{\alpha}-\sum_{i \in I_{\alpha}}d_{F_{i}}^{+}(x)) \\
\leq & \sum_{j=1}^{t}\sum_{i \in I}|P_{i}(X_{j})|-\sum_{I_{\alpha} \subseteq I} (\sum_{i \in I_{\alpha}}d_{F^{*}_{i}}^{+}(x) -\sum_{i \in I_{\alpha}}d_{F_{i}}^{+}(x))  ~\text{(using  $\sum_{i \in I_{\alpha}}d^{+}_{F^{*}_{i}}(x) \leq c'_{\alpha}$)}  \\
= & \sum_{j=1}^{t}\sum_{i \in I}|P_{i}(X_{j})|-\sum_{i \in I}(d_{F^{*}_{i}}^{+}(x)-d_{F_{i}}^{+}(x)) \\
 = & \sum_{i \in I}(|\{X_{j}| X_{j} \cap V(F_{i}) =\emptyset \}|-(d_{F^{*}_{i}}^{+}(x)-d_{F_{i}}^{+}(x)))\\
\leq & \sum_{i \in I}|\{X_{j}| X_{j} \cap (V(F_{i}) \cup N_{F^{*}_{i}}^{+}(x))=\emptyset \}|\\ &(\text{since } | N_{F^{*}_{i}}^{+}(x) \setminus V(F_{i}) | =  d_{F^{*}_{i}}^{+}(x)-d_{F_{i}}^{+}(x) ) 
\end{split}
\end{equation}
\begin{align*}
\begin{split}
\leq & \sum_{i \in I}\sum_{j=1}^{t}d^{-}_{F^{*}_{i} \setminus (F_{i} \cup E^{+}(x) )}(X_{j}) 
 =  \sum_{j=1}^{t}\sum_{i \in I}d^{-}_{F^{*}_{i} \setminus ( F_{i} \cup E^{+}(x))}(X_{j})\\
= & \sum_{j=1}^{t}d_{\cup_{i\in I} (F_{i}^{*} \setminus (F_{i} \cup E^{+}(x))) }^{-}(X_{j}) 
\leq  \sum_{j=1}^{t}d_{A \setminus (\cup_{i=1}^{k}F_{i} \cup E^{+}(x))}^{-}(X_{j}).
\end{split}
\end{align*}

\noindent \textbf{($\Leftarrow$) Sufficiency:} 
Suppose $\lambda$ is the number  of $\alpha \in [l]$ such that $\sum_{i \in I_{\alpha}}|d_{F_{i}}^{+}(x)|< c'_{\alpha}$.
We prove the sufficiency by induction on $\lambda$.
In the proof, {\em whenever we say $I \subseteq [k]$, we always mean that $I$
	is a union of some elements in $\{I_{1}, \ldots, I_{l}\}$.}

Let $I $ be a nonempty subset of $[k]$, and $\mathcal{F}$ be a multiset  consisting of some subsets of $V$ (these subsets do not have to be different). 
Define  
\[
\begin{split}
H(I, \emptyset) & := 0; \\
H(I, \mathcal{F}) &:= \sum_{X \in \mathcal{F}}(|P_{I}(X)|-d_{{A \setminus (\cup_{i=1}^{k}F_{i} \cup E^{+}(x))}}^{-}(X) );  %~~~~ \\
\end{split}
\]
\[
\begin{split}
\mathcal{E}^{1}_{I} := \{ \mathcal{F} \in D(V) : ~ & H(I,\mathcal{F} )=\sum_{I_{\alpha} \subseteq I}(c'_{\alpha} - \sum_{i \in I_{\alpha}}d^{+}_{F_{i}}(x)), \\
& \text{and } |P_{I}(X)|-d_{{A \setminus (\cup_{i=1}^{k} F_{i} \cup E^{+}(x))}}^{-}(X)>0, \text{ for } X \in \mathcal{F} \};
\end{split}
\]
\[
\begin{split}
\mathcal{E}^{2} :=  \{ \mathcal{F} \in D(V): ~ & H([k], \mathcal{F})=\sum_{X \in \mathcal{F}}[x, X]_{A \setminus \cup_{i=1}^{k}F_{i}},  \\
  & \text{and } |P(X)|-d_{{A \setminus (\cup_{i=1}^{k}F_{i} \cup E^{+}(x))}}^{-}(X)>0, \text{ for } X \in \mathcal{F} \}.
\end{split}
\]
Note that {$d_{A \setminus \cup _{i=1}^{k}F_{i}}^{-}(X)=d_{A \setminus (\cup_{i=1}^{k}F_{i} \cup E^{+}(x))}^{-}(X) + [x, X]_{A \setminus \cup_{i=1}^{k}F_{i}} $}; for $X \in \mathcal{F} \in \mathcal{E}^{2}$, by $(\ref{11})$, $ |P(X)|-d_{{A \setminus (\cup_{i=1}^{k}F_{i} \cup E^{+}(x))}}^{-}(X) \leq [x, X]_{A \setminus \cup_{i=1}^{k}F_{i}}$; thus
%by the definition of $H([k], \mathcal{F})$,
$H([k], \mathcal{F}) \leq \sum_{X \in \mathcal{F}}[x, X]_{A \setminus \cup_{i=1}^{k}F_{i}}$. Since $ \mathcal{F} \in \mathcal{E}^{2}$, we have  $H([k], \mathcal{F})=\sum_{X \in \mathcal{F}}[x, X]_{A \setminus \cup_{i=1}^{k}F_{i}}$. It follows that $|P(X)|-d_{{A \setminus (\cup_{i=1}^{k}F_{i} \cup E^{+}(x))}}^{-}(X) = [x, X]_{A \setminus \cup_{i=1}^{k}F_{i}}$.

\noindent\textbf{Process of PIEOs.} Suppose $I \subseteq [k]$,    $\mathcal{F}_{1} \in \mathcal{E}_{I}^{1}$ and $\mathcal{F}_{2} \in \mathcal{E}_{I}^{1} \cup \mathcal{E}^{2}$. Let  $\mathcal{G}_{0}=\mathcal{F}_{1} \uplus \mathcal{F}_{2}$. If there exists a {properly} intersecting pair $X$ and $Y$ in $\mathcal{G}_{0}$.
Since both  $\mathcal{F}_{1}, \mathcal{F}_{2} \in \mathcal{E}^{1}_{I} \cup \mathcal{E}^{2}$, we have $|P_{I}(X)|-d_{{A \setminus (\cup_{i=1}^{k}F_{i} \cup E^{+}(x))}}^{-}(X), |P_{I}(Y)|-d_{{A \setminus (\cup_{i=1}^{k}F_{i} \cup E^{+}(x))}}^{-}(Y) > 0$.
Recall that function $|P_{I}|- d^{-}$ is intersecting supermodular, we have
\[
\begin{split}
0 & <|P_{I}(X)|-d_{{A \setminus (\cup_{i=1}^{k}F_{i} \cup E^{+}(x))}}^{-}(X) + |P_{I}(Y)|-d_{{A \setminus (\cup_{i=1}^{k}F_{i} \cup E^{+}(x))}}^{-}(Y) \\
& \leq |P_{I}(X \cup Y)|-d_{{A \setminus (\cup_{i=1}^{k}F_{i} \cup E^{+}(x))}}^{-}(X \cup Y) + |P_{I}(X \cap Y)|-d_{{A \setminus (\cup_{i=1}^{k}F_{i} \cup E^{+}(x))}}^{-}(X \cap Y).
\end{split}
\] 
This gives $|P_{I}(X\cup Y)|-d_{{A \setminus (\cup_{i=1}^{k}F_{i} \cup E^{+}(x))}}^{-}(X \cup Y) >0 $ or $|P_{I}(X \cap Y)|-d_{{A \setminus (\cup_{i=1}^{k}F_{i} \cup E^{+}(x))}}^{-}(X \cap Y) >0$. Then we obtain $\mathcal{G}_{1}$ by replacing $X$ and $Y $ in $\mathcal{G}_{0}$ with the following subsets: $(i)$ $X \cup Y$, if $ |P_{I}(X \cap Y)|-d_{{A \setminus (\cup_{i=1}^{k}F_{i} \cup E^{+}(x))}}^{-}(X \cap Y) \leq 0$; $(ii)$ $X \cap Y$, if $ |P_{I}(X \cup Y)|-d_{{A \setminus (\cup_{i=1}^{k}F_{i} \cup E^{+}(x))}}^{-}(X \cup Y) \leq 0$; or $(iii)$ $X \cup Y$ and $X \cap Y$, otherwise. Suppose we adopt PIEOs in $\mathcal{G}_{0}$ in this way, step by step, and obtain families of subsets of $V$, $\mathcal{G}_{0}, \ldots, \mathcal{G}_{n}$, until there are no {properly} intersecting pairs anymore. Let $\mathcal{G}_{i}'$ be the family of maximal elements in $\mathcal{G}_{i}$ for $0 \leq i \leq n$. Let $\mathcal{F}_{3}:= \mathcal{G}_{n}' $ and $\mathcal{F}_{4}:=\mathcal{G}_{n} \setminus \mathcal{F}_{3} $. If $\mathcal{G}_{0}$ is laminar, then $n=0$ and it is not hard to see that $\mathcal{F}_{3}=\mathcal{F}_{1} \vee \mathcal{F}_{2}$ and $\mathcal{F}_{4}=\mathcal{F}_{1} \wedge \mathcal{F}_{2}$. By the constructing process of $\mathcal{G}_{n}$, we have
\begin{equation} \label{49}
H(I,\mathcal{F}_{1})+  H(I,\mathcal{F}_{2})=H(I, \mathcal{G}_{0})  \leq \ldots \leq  H(I, \mathcal{G}_{n})= H(I, \mathcal{F}_{3}) +  H(I, \mathcal{F}_{4}).
\end{equation}

\begin{claim}\label{48}
	Suppose $I \subseteq [k]$,  and  $\mathcal{F}_{1}, \mathcal{F}_{2} \in \mathcal{E}_{I}^{1}$. Then $\mathcal{F}_{3}, \mathcal{F}_{4} \in \mathcal{E}^{1}_{I}$.
\end{claim}

\begin{proof}
	Since $\mathcal{F}_{1}, \mathcal{F}_{2} \in \mathcal{E}_{I}^{1}$ and (\ref{49}) holds,
	\[
	2 \sum_{I_{\alpha} \subseteq I}(c'_{\alpha} - \sum_{i \in I_{\alpha}}d_{F_{i}}^{+}(x)) =  H(I, \mathcal{F}_{1}) +  H(I, \mathcal{F}_{2}) \leq  H(I, \mathcal{F}_{3}) +  H(I, \mathcal{F}_{4}).
	\]
	Since (\ref{22}) implies both
	\[
	H(I,\mathcal{F}_{3}) \leq \sum_{I_{\alpha} \subseteq I}(c'_{\alpha} - \sum_{i \in I_{\alpha}}d_{F_{i}}^{+}(x))
	%\]
	%\[
	~~\mbox{ and } ~~	
	 H(I, \mathcal{F}_{4} ) \leq \sum_{I_{\alpha} \subseteq I}(c'_{\alpha} - \sum_{i \in I_{\alpha}}d_{F_{i}}^{+}(x)).
	\]
	Therefore,
	\[
	H(I,\mathcal{F}_{3})= H(I, \mathcal{F}_{4} )= \sum_{I_{\alpha} \subseteq I}(c'_{\alpha} - \sum_{i \in I_{\alpha}}d_{F_{i}}^{+}(x)).
	\]
	Combining the constructing process of $\mathcal{G}_{n}$, we have  $\mathcal{F}_{3}$, $\mathcal{F}_{4} \in \mathcal{E}^{1}_{I}$.
\end{proof}

\begin{claim} \label{20}
	Let $I \subseteq [k]$. Suppose $\mathcal{V}_{I} \in \mathcal{E}^{1}_{I}$ satisfies: $(i)$ $\cup \mathcal{V}_{I}$ is minimal; and $(ii)$ subject to (i), $|\mathcal{V}_{I}|$ is maximum. Then $\mathcal{V}_{I}$ is minimum 
	in $(\mathcal{E}^{1}_{I}, \leq )$.
\end{claim} 

\begin{proof}
	It suffices to show for any $\mathcal{F}_{0} \in \mathcal{E}^{1}_{I}$, we have  $\mathcal{V}_{I} \leq \mathcal{F}_{0}$.

	If $\mathcal{G}_{0}$ is laminar,
	then $\mathcal{F}_{0} \wedge \mathcal{V}_{I} = \mathcal{F}_{4} \in \mathcal{E}^{1}_{I}$ (by Claim~\ref{48}). Since $\cup ( \mathcal{F}_{0} \wedge \mathcal{V}_{I} ) \subseteq \cup \mathcal{V}_{I}$ and $\cup \mathcal{V}_{I}$ is minimal, thus $\cup (\mathcal{F}_{0} \wedge \mathcal{V}_{I} ) = \cup \mathcal{V}_{I}$. Combining $ \mathcal{F}_{0} \wedge \mathcal{V}_{I} \leq \mathcal{V}_{I}$, for any $ X_{1} \in \mathcal{V}_{I}$, we have $X_{1} = \cup \{X_{2} \in \mathcal{F}_{0} \wedge \mathcal{V}_{I}: X_{2} \subseteq X_{1} \}$; thus $ |\mathcal{F}_{0} \wedge \mathcal{V}_{I}| \geq |\mathcal{V}_{I} |$.
	By the choice $(ii)$ of $\mathcal{V}_{I}$, $  |\mathcal{F}_{0} \wedge \mathcal{V}_{I}|  \leq |\mathcal{V}_{I} |$. So $ |\mathcal{V}_{I} \wedge \mathcal{F}_{0}| = |\mathcal{V}_{I} |$. And  
	for any  $ X_{1} \in \mathcal{V}_{I}$, $| \{X_{2} \in \mathcal{F}_{0} \wedge \mathcal{V}_{I}: X_{2} \subseteq X_{1} \}|=1$, that is $X_{1} \in \mathcal{F}_{0} \wedge \mathcal{V}_{I}$. Hence, $ \mathcal{F}_{0} \wedge \mathcal{V}_{I} = \mathcal{V}_{I} $.  This proves  $\mathcal{V}_{I} \leq \mathcal{F}_{0}$. 
		
	Suppose $\mathcal{G}_{0}$ is not laminar.
	By Proposition~\ref{17-P} (let $\mathcal{F}_{1}:=\mathcal{F}_{0}$ and $\mathcal{F}_{2}:=\mathcal{V}_{I}$), $\cup \mathcal{F}_{4} \subseteq  (\cup \mathcal{F}_{0}) \cap (\cup \mathcal{V}_{I}) $. Since $\mathcal{F}_{4} \in \mathcal{E}_{I}^{1}$ (by Claim~\ref{48}) and $\cup \mathcal{V}_{I}$ is minimal, $\cup \mathcal{F}_{4} = \cup \mathcal{V}_{I}$. Thus $\cup \mathcal{F}_{4} = (\cup \mathcal{F}_{0}) \cap (\cup \mathcal{V}_{I}) $); and by the choice $(ii)$ of $\mathcal{V}_{I}$, $ |\mathcal{F}_{4}| \leq |\mathcal{V}_{I}|$.
	
	Since $\cup \mathcal{F}_{4} = (\cup \mathcal{V}_{I}) \cap (\cup \mathcal{F}_{0}) $), by Proposition~\ref{17-P}, for any $i \in [n]$, $\mathcal{G}_{i-1} \xrightarrow{1} \mathcal{G}_{i}$. Then for $i \in [n]$, $|\mathcal{G}_{i-1}|=|\mathcal{G}_{i}|$, and thus $ |\mathcal{F}_{0}|+|\mathcal{V}_{I}|= |\mathcal{G}_{0}|=|\mathcal{G}_{n}|= |\mathcal{F}_{3}| + |\mathcal{F}_{4}|$. Combining  $ |\mathcal{F}_{4}| \leq |\mathcal{V}_{I}|$, 
	we have $|\mathcal{F}_{0}| \leq |\mathcal{F}_{3}|$. According to Proposition~\ref{18-P} $(ii)$, $ |\mathcal{F}_{3}| \leq |\mathcal{F}_{0}|$. Thus $ |\mathcal{F}_{3}| = |\mathcal{F}_{0}|$. Then by Proposition~\ref{18-P} (ii), we have $\mathcal{V}_{I} \leq \mathcal{F}_{0}$.
	%\QEDA
\end{proof}

For $I \subseteq [k]$, by Claim \ref{20},  
we suppose $\mathcal{V}_{I} \in \mathcal{E}^{1}_{I}$ is minimum in $(\mathcal{E}^{1}_{I}, \leq )$. 

\begin{claim}\label{45}
	Suppose $I \subseteq [k]$,   and $\mathcal{F}_{1}, \mathcal{F}_{2} \in \mathcal{E}_{I}^{1}$. Then
	\begin{itemize}
		\item[(i)] for any $X_{1} \in \mathcal{F}_{i}$, $i=1,2$, there exists $X_{2} \in \mathcal{V}_{I} $ such that $X_{2} \subseteq X_{1}$;
		
		\item[(ii)] $\mathcal{F}_{1} \vee \mathcal{F}_{2} \in \mathcal{E}_{I}^{1}$.
	\end{itemize}
\end{claim}

\begin{proof}
	$(i)$ Suppose, to the contrary, without loss of generality, there exists $X_{0} \in \mathcal{F}_{1}$ such that $X_{0} \cap (\cup \mathcal{V}_{I}) = \emptyset$ (otherwise, there exists $Y \in  \mathcal{V}_{I}$, such that $X_{0} \cap Y \neq  \emptyset$; since both  $\mathcal{F}_{1},  \mathcal{V}_{I}  \in \mathcal{E}_{I}^{1} \subseteq D(V)$, and $\mathcal{V}_{I}$ is minimum in $(\mathcal{E}^{1}_{I}, \leq )$, we deduce that $Y \subseteq X_0$.).
	Consider $\mathcal{V}_{I}+X_{0}$ and $\mathcal{F}_{1}-X_{0}$. Since
	\[
	H(I, \mathcal{V}_{I}+X_{0}) + H(I, \mathcal{F}_{1}-X_{0})=H(I, \mathcal{V}_{I})+H(I, \mathcal{F}_{1})=2 \sum_{I_{\alpha} \subseteq I} (c'_{\alpha} - \sum_{i \in I_{\alpha}} d_{F_{i}}^{+}(x) ),
	\]
	where the last equality is due to that both $\mathcal{V}_{I}, \mathcal{F}_{1}  \in \mathcal{E}_{I}^{1}$.
	By (\ref{22}),
	\[
	H(I, \mathcal{V}_{I}+X_{0}), H(I, \mathcal{F}_{1}-X_{0}) \leq  \sum_{I_{\alpha} \subseteq I} (c'_{\alpha} - \sum_{i \in I_{\alpha}} d_{F_{i}}^{+}(x) ).
	\]
	We deduce that $H(I, \mathcal{V}_{I}+X_{0}) = \sum_{I_{\alpha} \subseteq I} (c'_{\alpha} - \sum_{i \in I_{\alpha}} d_{F_{i}}^{+}(x) )$.
	Since $\mathcal{V}_{I} \in \mathcal{E}_{I}^{1} $, we have $H(I,\mathcal{V}_{I} )=\sum_{I_{\alpha} \subseteq I}(c'_{\alpha} - \sum_{i \in I_{\alpha}}d^{+}_{F_{i}}(x))$. Thus $H(I, \mathcal{V}_{I}+X_{0}) = H(I,\mathcal{V}_{I} )$; this gives us $|P_{I}(X_{0})|-d_{{A \setminus (\cup_{i=1}^{k}F_{i} \cup E^{+}(x))}}^{-}(X_{0})=0$. On the other hand,  since $X_{0} \in \mathcal{F}_{1} \in \mathcal{E}_{I}^{1}$, we have $ |P_{I}(X_{0})|-d_{{A \setminus (\cup_{i=1}^{k}F_{i} \cup E^{+}(x))}}^{-}(X_{0})>0$. This is a contradiction.

	$(ii)$   If $\mathcal{G}_{0}$ is laminar, then by Claim~\ref{48},   $ \mathcal{F}_{1} \vee \mathcal{F}_{2} = \mathcal{F}_{3} \in \mathcal{E}_{I}^{1}$. Suppose $\mathcal{G}_{0}$ is not laminar. %Next w
	We claim   that for all $i \in [n]$, $\mathcal{G}_{i-1} \xrightarrow{1~or~2} \mathcal{G}_{i}$.
	Then by Proposition~\ref{18-P} $(i)$ and 
	Claim~\ref{48}, $\mathcal{F}_{1} \vee \mathcal{F}_{2} = \mathcal{F}_{3}  \in \mathcal{E}^{1}_{I}$. This will finish the proof of $(ii)$.
	
	Suppose, to the contrary, there exists a minimum $i_{0} \in [n]$ such that $\mathcal{G}_{i_{0}-1} \xrightarrow{3} \mathcal{G}_{i_{0}}$;  that is, $\mathcal{G}_{i-1} \xrightarrow{1~or~2} \mathcal{G}_{i}$ for $1 \leq i \leq i_{0}-1$, and we replace a {properly} intersecting pair $X_{1}$ and $Y_{1}$ in $\mathcal{G}_{i_{0}-1}$ with $ X_{1} \cap Y_{1}$ and obtain $\mathcal{G}_{i_{0}}$. Note that  by Proposition~\ref{47-P}, $X_{1}, Y_{1} \in \mathcal{G}_{i_{0}-1}'$. By Proposition~\ref{52-P}, there exists $X_{2} \in \mathcal{G}_{0}$ such that $ X_{2} \subseteq X_{1}$; by  $(i)$ of this claim, 
	there exists $X_{3} \in \mathcal{V}_{I}$ such that $X_{3} \subseteq X_{2}$.
	Let $v \in X_{3}$. Since $\mathcal{G}_{i_{0}-1} \xrightarrow{3} \mathcal{G}_{i_{0}}$ and $ X_{3} \subseteq X_{1} \cup Y_{1}$, we have
	$  \mathcal{G}_{i_{0}}(v) < \mathcal{G}_{i_{0}-1}(v)$.
	By Proposition~\ref{clm-i-n-PIE-P}, we have $\mathcal{G}_{i_{0}-1}(v) \leq 2$, and then $\mathcal{F}_{3}(v) + \mathcal{F}_{4}(v)= \mathcal{G}_{n}(v) \leq \mathcal{G}_{i_{0}}(v) \leq 1$. Since $ \mathcal{F}_{4}(v) \le  \mathcal{F}_{3}(v)$,
	%(then $\cup \mathcal{F}_{4}^{2} \subseteq \cup \mathcal{F}^{2}_{3}$),
	we have $\mathcal{F}_{4}(v)=0$. However,  $v \in X_{3} \in \mathcal{V}_{I} \leq \mathcal{F}_{4}$, a contradiction.
\end{proof}

For $I \subseteq [k]$, by Claim~\ref{45} $(ii)$,  we suppose  $\mathcal{U}_{I} \in \mathcal{E}^{1}_{I}$ is maximum in $(\mathcal{E}^{1}_{I}, \leq)$.

\begin{claim} \label{33}
	Suppose $\mathcal{F}_{1} \in \mathcal{E}_{[k]}^{1}$ and $\mathcal{F}_{2} \in \mathcal{E}^{2}$. Then $\mathcal{F}_{1} \vee \mathcal{F}_{2} \in \mathcal{E}^{1}_{[k]}$.
\end{claim}

\begin{proof}
	$H([k], \mathcal{F}_{4}) =$ 
	$\sum_{X \in \mathcal{F}_4}(|P(X)|-d_{{A \setminus (\cup_{i=1}^{k}F_{i} \cup E^{+}(x))}}^{-}(X) )$
	$\leq \sum_{X \in \mathcal{F}_{4}} [x, X]_{A \setminus \cup_{i=1}^{k}F_{i}} = [x, \cup\mathcal{F}_{4}]_{A \setminus \cup_{i=1}^{k}F_{i}}$,
	%\]
	where the inequality is due to (\ref{11}).
	Since $\cup \mathcal{F}_{4} \subseteq (\cup \mathcal{F}_{1}) \cap (\cup \mathcal{F}_{2} )$ (by Proposition~\ref{17-P}) and $\mathcal{F}_{2} \in  \mathcal{E}^{2}$,
	\begin{equation}\label{51}
	H([k], \mathcal{F}_{4}) \leq [x, \cup\mathcal{F}_{4}]_{A \setminus \cup_{i=1}^{k}F_{i}} \leq [x, \cup\mathcal{F}_{2}]_{A \setminus \cup_{i=1}^{k}F_{i}}=H([k],\mathcal{F}_{2}).
	\end{equation}
	
	By (\ref{22}), combining $\mathcal{F}_{1} \in \mathcal{E}_{[k]}^{1}$,   we have
	\begin{equation}\label{50}
	H([k], \mathcal{F}_{3}) \leq   \sum_{\alpha=1}^{l}(c'_{\alpha} -\sum_{i \in I_{\alpha}}d^{+}_{F_{i}}(x)) =H([k], \mathcal{F}_{1}) .
	\end{equation}
	
	Then (\ref{51}) and (\ref{50}) give
	\[
	H([k], \mathcal{F}_{3})+ H([k], \mathcal{F}_{4}) \leq H([k], \mathcal{F}_{1})+ H([k], \mathcal{F}_{2}).
	\]
	By (\ref{49}), we have
	\[
	H([k], \mathcal{F}_{1})+ H([k], \mathcal{F}_{2})=  H([k], \mathcal{F}_{3})+ H([k], \mathcal{F}_{4}),
	\]
	%and all $''\leq''$ of (\ref{51}) and (\ref{50}) should be $''=''$.
	and the $``\leq"$s of  (\ref{51}) and (\ref{50}) should be $``="$s.
	So
	\begin{equation} \label{H-4-2}
	[x, \cup\mathcal{F}_{4}]_{A \setminus \cup_{i=1}^{k}F_{i}} = [x, \cup\mathcal{F}_{2}]_{A \setminus \cup_{i=1}^{k}F_{i}}
	\end{equation}
	and
	$H([k], \mathcal{F}_{3}) = \sum_{\alpha=1}^{l}(c'_{\alpha} -\sum_{i \in I_{\alpha}}d^{+}_{F_{i}}(x))$. Combining the constructing process of $\mathcal{G}_{n}$, we have $ \mathcal{F}_{3} \in \mathcal{E}_{[k]}^{1}$.
	
	If $\mathcal{G}_{0}$ is laminar, then   $ \mathcal{F}_{1} \vee \mathcal{F}_{2} = \mathcal{F}_{3} \in \mathcal{E}_{[k]}^{1}$. Suppose $\mathcal{G}_{0}$ is not laminar. 
	We claim   that for all $i \in [n]$, $\mathcal{G}_{i-1} \xrightarrow{1~or~2} \mathcal{G}_{i}$.
	Then by Proposition~\ref{18-P} $(i)$,  $\mathcal{F}_{1} \vee \mathcal{F}_{2} = \mathcal{F}_{3}  \in \mathcal{E}^{1}_{[k]}$.

	Suppose otherwise, there exists a minimum $i_{0} \in [n]$ such that $\mathcal{G}_{i_{0}-1} \xrightarrow{3} \mathcal{G}_{i_{0}}$. That is, $\mathcal{G}_{i-1} \xrightarrow{1~or~2} \mathcal{G}_{i}$ for $1 \leq i \leq i_{0}-1$, and we replace a {properly} intersecting pair $X_{4}$ and $Y_{4}$ in $\mathcal{G}_{i_{0}-1}$ with $ X_{4} \cap Y_{4}$ and obtain $\mathcal{G}_{i_{0}}$. Note that  by Proposition~\ref{47-P}, $X_{4}, Y_{4} \in \mathcal{G}_{i_{0}-1}'$. By Proposition~\ref{52-P}, there exists $X_{5} \in \mathcal{F}_{2}$ such that $ X_{5} \subseteq X_{4} \cup Y_{4}$.
	For $v \in X_{5}$, since $\mathcal{G}_{i_{0}-1} \xrightarrow{3} \mathcal{G}_{i_{0}}$ and $ X_{5} \subseteq X_{4} \cup Y_{4}$,
	$  \mathcal{G}_{i_{0}}(v) < \mathcal{G}_{i_{0}-1}(v)$.
	By Proposition~\ref{clm-i-n-PIE-P}, $\mathcal{G}_{i_{0}-1}(v) \leq 2$, and thus $\mathcal{F}_{3}(v) + \mathcal{F}_{4}(v)= \mathcal{G}_{n}(v) \leq \mathcal{G}_{i_{0}}(v) \leq 1$. Since $ \cup \mathcal{F}_{4} \subseteq \cup \mathcal{F}_{3}$, $\mathcal{F}_{4}(v)=0$. So $ X_{5} \cap (\cup \mathcal{F}_{4})= \emptyset $; by Proposition \ref{17-P}, $\cup\mathcal{F}_{4} \subseteq (\cup\mathcal{F}_{2}) \setminus X_{5}$. Recall that since $X_{5} \in \mathcal{F}_{2} \in \mathcal{E}^{2}$, $ |P(X_{5})|-d_{{A \setminus (\cup_{i=1}^{k}F_{i} \cup E^{+}(x))}}^{-}(X_{5}) = [x, X_{5}]_{A \setminus \cup_{i=1}^{k}F_{i}} > 0$; thus $ [x, \bigcup \mathcal{F}_{4}]_{A \setminus \cup_{i=1}^{k}F_{i}} = [x, \cup \mathcal{F}_{2}]_{A \setminus \cup_{i=1}^{k}F_{i}}-[x, X_{5}]_{A \setminus \cup_{i=1}^{k}F_{i}} <[x, \bigcup \mathcal{F}_{4}]_{A \setminus \cup_{i=1}^{k}F_{i}}$, a contradiction to (\ref{H-4-2}).
\end{proof}

\begin{claim} \label{34}
	Suppose $X \in \mathcal{U}_{[k]}$ and $|P(Y)|-d_{{A \setminus (\cup_{i=1}^{k}F_{i} \cup E^{+}(x))}}^{-}(Y)=[x, Y]_{A \setminus \cup_{i=1}^{k}F_{i}}>0$. If $X \cap Y \neq \emptyset$, then $Y \subseteq X$.
\end{claim}

%\proof
\begin{proof}
	Suppose otherwise, that is $Y \setminus X \neq \emptyset$. Let $\mathcal{F}_{1}=\mathcal{U}_{[k]}$ and $\mathcal{F}_{2}=\{ Y\}$. Then $ \mathcal{F}_{2} \nleq \mathcal{F}_{1}$ and thus $\mathcal{U}_{[k]} = \mathcal{F}_{1} < \mathcal{F}_{1} \vee \mathcal{F}_{2}$. By Claim~\ref{33}, $\mathcal{F}_{1} \vee \mathcal{F}_{2} \in \mathcal{E}^{1}_{[k]}$, contradicting that $\mathcal{U}_{[k]}$ is maximum in $\mathcal{E}^{1}_{[k]}$.
	%\QEDA
\end{proof}

The main idea of the proof for the sufficiency of Theorem~\ref{31}
is
(i) to find an $\alpha_{0} \in [l]$ such that (a) $\sum_{i \in I_{\alpha_{0}}}d_{F_{i}}^{+}(x) <c'_{\alpha_{0}}$,
and (b) after we update $c'_{\alpha_{0}}:=c'_{\alpha_{0}}-1$, (\ref{22}) still holds;
or (ii) to find an $\alpha_{0} \in [l]$, $i_{0} \in I_{\alpha_{0}}$ and $e_{_{0}} \in E^{+}(x) \setminus \cup_{i=1}^{k}F_{i}$ such that
$(a)$ $\sum_{i \in I_{\alpha_{0}}}d^{+}_{F_{i}}(x) < c'_{\alpha_{0}}$,
$(b)$ $h(e_{0}) \notin V(F_{i_{0}})$, and $(c)$ after we update  $F_{i_{0}}:=F_{i_{0}}+e_{0}$,  (\ref{11}) and (\ref{22}) still hold.  Continue this process until $\sum_{i \in I_{\alpha}}d_{F_{i}}^{+}(x)=c'_{\alpha}$ for all $\alpha \in [l]$; meanwhile, (\ref{22}) still holds. Then (\ref{22}) implies $|P(X)| \leq d_{{A \setminus (\cup_{i=1}^{k}F_{i} \cup E^{+}(x))}}^{-}(X)$ for $\emptyset \neq X \subseteq V$ by setting $I=[k]$ and $t=1$. So $|P(V)| \leq d_{A \setminus (\cup_{i=1}^{k}F_{i} \cup E^{+}(x) )}^{-}(V)=0$; and thus $P(V)=\emptyset$, that is $V(F_{i})-x \neq \emptyset$ for $i \in [k]$. By Theorem~\ref{16}, there exist arc-disjoint branchings $B_{i}$ with root set $V(F_{i})-x$ in $D-\cup_{i=1}^{k}F_{i}-x$ for $1 \leq i \leq k$. Finally, we obtain $F^{*}_{1}=F_{1} \cup B_{1}, \ldots, F^{*}_{k}=F_{k} \cup B_{k}$ as demanded.

The proof 
is {\bf by induction on the number $\lambda$ of $\alpha \in [l]$ such that $\sum_{i \in I_{\alpha}}d^{+}_{F_{i}}(x) < c'_{\alpha}$.}
For the base step $\lambda=0$, as explained above.  
For the induction step, suppose $\lambda \geq 1$; and without loss of generality, $\sum_{i \in I_{\alpha}}d_{F_{i}}^{+}(x)< c¡¯_{\alpha}$ for $1 \leq \alpha \leq \lambda$.

\begin{observation} \label{21}
	Let $\alpha_{0} \in [l]$, $i_{0} \in I_{\alpha_{0}}$ and $e_{_{0}} \in E^{+}_{A\setminus \cup_{i=1}^{k}F_{i} }(x) $ such that $h(e_{0}) \notin V(F_{i_{0}})$. Let $F_{i_{0}}:=F_{i_{0}} +e_{0}$ and $I \subseteq [k]$. Then
	\begin{itemize}
		\item[(i)] if $ I \nsupseteq I_{\alpha_{0}}$, (\ref{22}) still holds;  
		\item[(ii)] if $ I \supseteq I_{\alpha_{0}}$, $\sum_{j=1}^{t}(|P_{I}(X_{j})|-d^{-}_{{A \setminus (\cup_{i=1}^{k}F_{i} \cup E^{+}(x))}}(X_{j}))-\sum_{I_{\alpha} \subseteq I}(c'_{\alpha}-\sum_{i \in I_{\alpha}} d_{F_{i}}^{+}(x))$ is increased by at most $1$.
	\end{itemize}
\end{observation}

\begin{claim} \label{15}
	Suppose $ I \supseteq \cup_{\alpha=1}^{\lambda}I_{\alpha}  $. Then  $\mathcal{E}^{1}_{I} \subseteq \mathcal{E}^{1}_{[k]}$.
\end{claim}
\begin{proof}
	Let $\mathcal{F} \in \mathcal{E}^{1}_{I}$. Since
     $|P_{I}(X)| \leq |P_{[k]}(X)|$ for $X \in \mathcal{F}$,
	$\sum_{\alpha=1}^{\lambda}(c'_{\alpha} - \sum_{i \in I_{\alpha}}d_{F_{i}}^{+}(x))=H(I, \mathcal{F}) \leq H([k], \mathcal{F})$. By (\ref{22}), $H([k], \mathcal{F}) \leq  \sum_{\alpha=1}^{\lambda}(c'_{\alpha} - \sum_{i \in I_{\alpha}}d_{F_{i}}^{+}(x))$. Hence, $H([k], \mathcal{F}) =  \sum_{\alpha=1}^{\lambda}(c'_{\alpha} - \sum_{i \in I_{\alpha}}d_{F_{i}}^{+}(x))$, that is, $\mathcal{F} \in \mathcal{E}^{1}_{[k]}$.
\end{proof}

\begin{claim}\label{35}
	Suppose $\alpha_{0} \in [\lambda]$. Then there exists $X_{0} \in \mathcal{U}_{[k]}$ and $i_{0} \in I_{\alpha_{0}}$ such that $X_{0} \cap V(F_{i_{0}})=\emptyset$.
\end{claim}
\begin{proof}
	Suppose otherwise, for any $X \in \mathcal{U}_{[k]}$ and $i \in I_{\alpha_{0}}$, $X \cap V(F_{i}) \neq \emptyset$; thus $P_{[k]}(X)=P_{[k]\setminus I_{\alpha_{0}}}(X)$.
	Note that $\mathcal{U}_{[k]} \in \mathcal{E}^{1}_{[k]}$, and  $\sum_{i \in I_{\alpha_{0}}}d_{F_{i}}^{+}(x)<c'_{\alpha_{0}}$. Then
	\[
	H([k] \setminus I_{\alpha_{0}}, \mathcal{U}_{[k]})=H([k], \mathcal{U}_{[k]})= \sum_{\alpha=1}^{\lambda}(c'_{\alpha} - \sum_{i \in I_{\alpha}}d_{F_{i}}^{+}(x)) > \sum_{\alpha \in [\lambda]\setminus \{ \alpha_{0}\}}(c'_{\alpha} - \sum_{i \in I_{\alpha}}d_{F_{i}}^{+}(x)),
	\]
	However, by (\ref{22}), $H([k] \setminus I_{\alpha_{0}}, \mathcal{U}_{[k]}) \leq \sum_{\alpha \in [\lambda]\setminus \{ \alpha_{0}\}}(c'_{\alpha} - \sum_{i \in I_{\alpha}}d_{F_{i}}^{+}(x))$, a contradiction.
\end{proof}

\noindent {\bf Case 1} (of the induction step): Assume $\lambda=1$.

If $\mathcal{E}^{1}_{[k]} = \emptyset$, then by Claim~\ref{15}, for $ I \supseteq I_{1}$, $\mathcal{E}^{1}_{I} = \emptyset$.
%This gives us,
Next we show that, for $I \supseteq I_{1} $ and $\mathcal{F} \in D(V)$,
$H(I,\mathcal{F}) < c'_{1}-\sum_{i \in I_{1}}d_{F_{i}}^{+}(x)$.
Therefore, after we update $c'_{1}:=c'_{1}-1$,  $(\ref{22})$ still holds.

Suppose otherwise, that is, there exists $\mathcal{F}_{0} \in D(V)$ such that $H(I, \mathcal{F}_{0})=c'_{1}-\sum_{i \in I_{1}}d_{F_{i}}^{+}(x)$. Since $c'_{1}-\sum_{i \in I_{1}}d_{F_{i}}^{+}(x) >0$ and  $H(I, \emptyset)= 0$,
$\mathcal{F}_{1}=\{ X \in \mathcal{F}_0: |P_{I}(X)|-d_{{A \setminus (\cup_{i=1}^{k}F_{i} \cup E^{+}(x))}}^{-}(X) >0\} \neq \emptyset$. Since $H(I, \mathcal{F}_0 \setminus \mathcal{F}_{1}) \leq 0$,
\[
H(I, \mathcal{F}_{1}) \geq H(I, \mathcal{F}_{1})+H(I, \mathcal{F}_0 \setminus \mathcal{F}_{1})   = H(I, \mathcal{F}_0)=c'_{1}-\sum_{i \in I_{1}}d_{F_{i}}^{+}(x).
\]
On the other hand, by (\ref{22}), $H(I, \mathcal{F}_{1}) \leq c'_{1}-\sum_{i \in I_{1}}d_{F_{i}}^{+}(x)$. Hence, $H(I, \mathcal{F}_{1}) = c'_{1}-\sum_{i \in I_{1}}d_{F_{i}}^{+}(x)$. This shows that $\mathcal{F}_{1} \in \mathcal{E}_{I}^{1}$, a contradiction to $\mathcal{E}_{I}^{1}=\emptyset $.

Next, suppose $\mathcal{E}^{1}_{[k]} \neq \emptyset$. By Claim~\ref{35}, there exists $X_{0} \in \mathcal{U}_{[k]}$ and $i_{0} \in I_{1}$ such that $X_{0} \cap V(F_{i_{0}})=\emptyset$. By Claim~\ref{45}, there exists $Y_{0} \in \mathcal{V}_{[k]}$ such that $Y_{0} \subseteq X_{0}$. 
Since $Y_{0} \in \mathcal{V}_{[k]} \in \mathcal{E}_{[k]}^{1}$ and (\ref{11}) holds, $0 < |P(Y_{0})|-d_{{A \setminus (\cup_{i=1}^{k}F_{i} \cup E^{+}(x))}}^{-}(Y_{0}) \leq [x, Y_{0}]_{A \setminus \cup_{i=1}^{k}F_{i}} $. So there exists $e_{0} \in E_{A \setminus \cup_{i=1}^{k}F_{i}}^{+}(x) $ with $h(e_{0}) \in Y_{0}$. Update $F_{i_{0}}:=F_{i_{0}}+e_{0}$. 

To show (\ref{11}) still holds, it suffices to show
%we can reduce the proof to that
for $X \subseteq V$ such that $h(e_{0}) \in X$ and $|P(X)|=d_{{A \setminus \cup_{i=1}^{k}F_{i}}}^{-}(X)$ before we update $F_{i_{0}}: =F_{i_{0}}+e_{0} $, $(\ref{11})$ still holds. Since $h(e_{0}) \in X_{0} \cap X$, $X \cap X_{0} \neq \emptyset$. By Claim~\ref{34}, $X \subseteq X_{0}$. Since $X_{0} \cap V(F_{i_{0}})=\emptyset$, $X \cap V(F_{i_{0}})=\emptyset$. Hence, both $d_{{A \setminus \cup_{i=1}^{k}F_{i}}}^{-}(X) $ and $|P(X)|$ are decreased by $1$, and thus (\ref{11}) still holds.

To show (\ref{22}) still holds, by Observation~\ref{21}, it suffices to show
%we can reduce the proof to that
for $I \supseteq I_{1}$ and $\mathcal{F} \in \mathcal{E}^{1}_{I}$, (\ref{22}) still holds. By Claim~\ref{15}, $ \mathcal{F} \in \mathcal{E}^{1}_{[k]}$. So $\mathcal{V}_{[k]} \leq \mathcal{F} \leq \mathcal{U}_{[k]}$, and there exists a $Z_{0} \in \mathcal{F}$ such that $Y_{0} \subseteq Z_{0} \subseteq X_{0}$. Since $X_{0} \cap V(F_{i_{0}}) =\emptyset$, $Z_{0} \cap V(F_{i_{0}})=\emptyset$ and $|P_{I}(Z_{0})|$ is decreased by $1$.
This cancels the increase of  $d_{F_{i_{0}}}^{+}(x)$, and thus (\ref{22}) still holds.
% is increased by $1$,

\noindent {\bf Case 2} (of the induction step): Assume $\lambda \geq 2$.

Suppose $p,q \in [\lambda]$ and $p \neq q$. Since (\ref{22}) holds for $\emptyset \neq I \subseteq [k]$ that is the union of some elements in $\{I_{1}, \ldots, I_{l}\}-I_{p}- I_{q}+(I_{p}\cup I_{q})$,
%$I_{p}\cup I_{q}, I_{1}, \ldots, I_{p-1}, I_{p+1}, \ldots, I_{q-1}, I_{q+1}, %\ldots, I_{l},$
by induction hypothesis, $F_{i}$ can be completed to spanning $x$-arborescence $F'_{i}$ for $1 \leq i \leq k$ such that $\sum_{i \in I_{p} \cup I_{q}}d_{F'_{i}}^{+}(x) \leq c'_{p}+c'_{q}$ and $\sum_{i \in I_{\alpha}} d_{F'_{i}}^{+}(x) \leq c'_{\alpha} $ for $\alpha \in [l]\setminus \{ p,q\}$.
Note that either  $\sum_{i \in I_{p}}d_{F'_{i}}^{+}(x) \leq c'_{p}$ or  $\sum_{i \in I_{q}}d_{F'_{i}}^{+}(x) \leq c'_{q}$.

\noindent {\bf Subcase 2.1} Assume $\sum_{i \in I_{p}}d_{F_{i}}^{+}(x) < \sum_{i \in I_{p}}d_{F'_{i}}^{+}(x) \leq c'_{p}$.

Then there exists $e_{0} \in E^{+}_{F'_{i_{0}} \setminus F_{i_{0}}}(x)$ for some $i_{0} \in I_{p}$. Update $F_{i_{0}}:=F_{i_{0}}+e_{0}$. Due to the existence of $F'_{1}, \ldots, F'_{k}$, by Theorem~\ref{10}, $(\ref{11})$ still holds. For $I \supseteq I_{p}$,  since $\sum_{i \in I}d_{F'_{i}}^{+}(x) \leq \sum_{I_{\alpha} \subseteq I}c'_{\alpha}$, (\ref{19}) holds and thus (\ref{22}) still holds. And by Observation~\ref{21} $(i)$, (\ref{22}) holds for $I \nsupseteq I_{p}$. Then we are done with this subcase.

\noindent {\bf Subcase 2.2} Assume $\sum_{i \in I_{q}}d_{F_{i}}^{+}(x) < \sum_{i \in I_{q}}d_{F'_{i}}^{+}(x) \leq c'_{q}$.

The proof is the same as  Subcase $2.1$.

\noindent {\bf Subcase 2.3 (the remaining case) } %(not Subcases $2.1$ and $2.2$)}
%For the left case,
We have either $\sum_{i \in I_{p}}d_{F_{i}}^{+}(x) = \sum_{i \in I_{p}}d_{F'_{i}}^{+}(x) \leq c'_{p}$ or $\sum_{i \in I_{q}}d_{F_{i}}^{+}(x) = \sum_{i \in I_{q}}d_{F'_{i}}^{+}(x) \leq c'_{q}$.
Recall $ \sum_{i \in I_{p}}d_{F_{i}}^{+}(x) <c'_{p}$ and $ \sum_{i \in I_{q}}d_{F_{i}}^{+}(x) <c'_{q}$.

If $\sum_{i \in I_{p}}d_{F_{i}}^{+}(x) = \sum_{i \in I_{p}}d_{F'_{i}}^{+}(x) < c'_{p}$, then for $ I_{p} \subseteq I \subseteq [k] \setminus I_{q}$, since $F_{i}$ can be completed to $F'_{i}$ for $i \in I$  such that
\[
\sum_{i \in I}d_{F'_{i}}^{+}(x) = \sum_{i \in I \setminus I_{p}}d_{F'_{i}}^{+}(x) + \sum_{i \in I_{p}}d_{F'_{i}}^{+}(x) \leq \sum_{I_{\alpha} \subseteq I, \alpha \neq p}c'_{\alpha}+\sum_{i \in I_{p}}d_{F_{i}}^{+}(x) < \sum_{I_{\alpha} \subseteq I}c'_{\alpha},
\]
(\ref{19}) holds but the equality of (\ref{19}) does not hold.
This gives for $ I_{p} \subseteq I \subseteq [k] \setminus I_{q}$,  $\mathcal{E}^{1}_{I}=\emptyset$.

If  $\sum_{i \in I_{q}}d_{F_{i}}^{+}(x) = \sum_{i \in I_{q}}d_{F'_{i}}^{+}(x) < c'_{q}$, use similar arguments as above.

We have  either  $(2.3.i)$~$\mathcal{E}^{1}_{I}=\emptyset$ for $ I_{p} \subseteq I \subseteq [k] \setminus I_{q}$ or $(2.3.ii)$~$\mathcal{E}^{1}_{I}=\emptyset$ for $ I_{q} \subseteq I \subseteq [k] \setminus I_{p}$.

\begin{claim}\label{23}
	There exists $\alpha_{0} \in [\lambda]$ such that  if  $I^* \supseteq I_{\alpha_{0}}$ and $\mathcal{E}^{1}_{I^*} \neq \emptyset$, then $I^* \supseteq \cup_{\alpha=1}^{\lambda}I_{\alpha} $.
\end{claim}

\begin{proof}
	Construct a tournament $T$ on vertex set $V(T)=[\lambda]$: for $p,q \in [\lambda]$, if $(2.3.i)$ holds, then let $pq \in A(T)$;
	else, $(2.3.ii)$ holds, then let $qp \in A(T)$. 
	Pick an $\alpha_{0} \in [\lambda]$ such that every vertex of $T$ is reachable from $\alpha_{0}$. Suppose $I^* \supseteq I_{\alpha_{0}}$ and $I^* \nsupseteq \cup_{\alpha=1}^{\lambda}I_{\alpha}$, let $R:=\{\alpha \in [\lambda] | I_{\alpha} \subseteq I^* \}$. Then $\alpha_{0} \in R \subsetneq [\lambda]$. Since every vertex in $[\lambda] \setminus R$ is reachable from $\alpha_{0}$, there exists an arc $st$ with $s \in R $ and $t \in [\lambda] \setminus R$. Since $I_{s} \subseteq I^* \subseteq [k]\setminus I_{t} $, by the definition of $A(T)$,
	%$ I_{\alpha_{0}}, $
	$\mathcal{E}^{1}_{I^*}=\emptyset$. Hence,  $I^* \supseteq I_{\alpha_{0}}$ and $\mathcal{E}^{1}_{I^*} \neq \emptyset$ implies $I^* \supseteq \cup_{\alpha=1}^{\lambda}I_{\alpha} $.
\end{proof}

Suppose $\alpha_{0}$ is as defined in Claim~\ref{23}.
For $I \supseteq I_{\alpha_{0}}$,  we  have  $\mathcal{E}^{1}_{I} \subseteq \mathcal{E}^{1}_{[k]}$.
%We only need to consider
To see this, consider
$I^* \supseteq I_{\alpha_{0}}$  such that
$\mathcal{E}^{1}_{I^*} \neq \emptyset$. Then by Claim~\ref{23}, $I^* \supseteq \cup_{\alpha=1}^{\lambda}I_{\alpha}$; by Claim~\ref{15}, $\mathcal{E}^{1}_{I^*} \subseteq \mathcal{E}^{1}_{[k]}$.

If $\mathcal{E}^{1}_{[k]} = \emptyset$, then $\mathcal{E}^{1}_{I} = \emptyset$ for $I \supseteq I_{\alpha_{0}}$. Update $c'_{\alpha_{0}}:=c'_{\alpha_{0}}-1$, and (\ref{22}) still holds. 

If $ \mathcal{E}^{1}_{[k]} \neq \emptyset$, by Claim~\ref{35}, there exists $X_{0} \in \mathcal{U}_{[k]}$ and $i_{0} \in I_{\alpha_{0}}$ such that $X_{0} \cap V(F_{i_{0}})=\emptyset$.
%
%Using
By the same arguments as in Case 1 ($\lambda=1$),
% (of the induction step):
%Exactly the same as the case $\lambda=1$,
there exist $Y_{0} \in \mathcal{V}_{[k]}$ such that $Y_{0} \subseteq X_{0}$;  
$e_{0} \in E_{A \setminus \cup_{i=1}^{k}F_{i}}^{+}(x) $ with $h(e_{0}) \in Y_{0}$;
after we update $F_{i_{0}}:= F_{i_{0}}+e_{0}$,  (\ref{11}) and (\ref{22}) still hold.
This finishes the proof of Subcase $2.3$, and  Theorem~\ref{31}.

\subsection{Branching packing with their root degrees bounded above }

As an application of Theorem~\ref{31}, we give a characterization for the existence of arc-disjoint $c^{-}$-branchings, whose root sets contain given vertices.

\begin{cor} \label{24}
	%Let $D$ be a digraph and $k$ be a positive integer.
	For digraph $D$ and integer $k>0$, let $I_{1},\ldots, I_{l}$ be a partition of $[k]$, $c'_{1}, \ldots, c'_{l}$ be nonnegative integers and $U_{1}, \ldots, U_{k}$ be subsets of $V(D)$ with $\sum_{i \in I_{\alpha}}|U_{i}| \leq c'_{\alpha}$ for $1 \leq \alpha \leq l$. Then there exist $k$ arc-disjoint branchings $B_{1}, \ldots, B_{k}$ in $D$ such that $U_{i} \subseteq R(B_{i}) $ for $1 \leq i \leq k$ and $\sum_{i \in I_{\alpha}} |R(B_{i})| \leq c'_{\alpha}$ for $1 \leq \alpha \leq l$ if and only if for any disjoint $X_{1}, \ldots, X_{t} \subseteq V(D)$ and any $I \subseteq [k]$ that is the union of some of $I_{1}, \ldots, I_{l}$,
	\begin{equation}\label{2}
	\sum_{j=1}^{t}(|P_{I}(X_{j})|-d_D^{-}(X_{j}))\leq \sum_{I_{\alpha} \subseteq I}(c'_{\alpha}-\sum_{i \in I_{\alpha}}|U_{i}|),
	\end{equation}
	where $ P_{I}(X)=\{ i \in I :  X \cap U_{i} =\emptyset\}. $
\end{cor}

Note that the definition of $ P_{I}(X)$ in Corollary~\ref{24} coincides with the original definition of $P_{I}(X)$ (explained next in the proof).

\begin{proof}
	We obtain a new digraph $D'$ from $D$ by adding a new vertex $x$ and $k $ parallel arcs from $x$ to each vertex in $V(D)$. Let $\{F_{i}\}_{i=1}^{k}$ be a family of arc-disjoint $x$-arborescences such that $V(F_{i})=U_{i}+x$ and $A(F_{i})$ consists of arcs from $x $ to each vertex in $U_{i}$ for $1 \leq i \leq k$. Then there exist arc-disjoint branchings $B_{1}, \ldots, B_{k}$ in $D$ such that $U_{i} \subseteq R(B_{i}) $ for $1 \leq i \leq k$ and $\sum_{i \in I_{\alpha}} |R(B_{i})| \leq c'_{\alpha}$ for $1 \leq \alpha \leq l$, if and only if in $D'$, $F_{1}, \ldots, F_{k}$ can be completed to $k$ arc-disjoint spanning $x$-arborescences $F^{*}_{1}, \ldots, F^{*}_{k}$  such that $\sum_{i \in I_{\alpha}}d^{+}_{F^{*}_{i}}(x) \leq c'_{\alpha}$ for $1 \leq \alpha \leq l$. By Theorem~\ref{31}, such $x$-arborescences exist if and only if (\ref{11}) and (\ref{22}) hold.
	Note that in $D'$, (\ref{11}) clearly holds, and (\ref{22}) is exactly (\ref{2}). This finishes the proof.
\end{proof}

In Corollary~\ref{24}, by setting $l=k$, $I_{i}=\{ i \}$ for $i \in [k]$, $U_{1}=\ldots =U_{k}=\emptyset$, we have Corollary \ref{BF},   which is first discovered by B\'{e}rczi and Frank \cite[Theorem 23]{berczi1}. 
Note that in our approach, it comes from arborescence augmentation, this is different  than  \cite{berczi1}.

\subsection{Proof of Theorem~\ref{32}}

The necessity is due to Theorem~\ref{1} and \ref{31}. %Next, we
We prove the sufficiency by induction on the number $\tau$ of $\alpha \in [l]$ such that $\sum_{i \in I_{\alpha}}d^{+}_{F_{i}}(x) < c_{\alpha} $. If $\tau=0$, then (\ref{4}) implies (\ref{11}) (by setting $t=1$ and $I=[k]$ in (\ref{4})), we are done by Theorem~\ref{31}.

Suppose $\tau \geq 1$, and $\alpha_{0} \in [l]$ with $\sum_{i \in I_{\alpha_{0}}} d^{+}_{F_{i}}(x) < c_{\alpha_{0}}$. By setting $c_{\alpha_{0}} := \sum_{i \in I_{\alpha_{0}}}d^{+}_{F_{i}}(x)$ and the induction hypothesis, $F_{1}, \ldots, F_{k}$ can be completed to arc disjoint spanning $x$-arborescence $F'_{1}, \ldots, F'_{k}$ such that $ c_{\alpha} \leq \sum_{i \in I_{\alpha}}d^{+}_{F'_{i}}(x) \leq c'_{\alpha}$ for $\alpha \in [l]\setminus \{ \alpha_{0}\} $ and $ \sum_{i \in I_{\alpha_{0}}}d^{+}_{F'_{i}}(x) \leq c'_{\alpha_{0}}$. If $\sum_{i \in I_{\alpha_{0}}}d^{+}_{F'_{i}}(x ) \geq c_{\alpha_{0}}$, then we are done. Otherwise, $\sum_{i \in I_{\alpha_{0}}}d^{+}_{F'_{i}}(x)$ $\leq$ $c_{\alpha_{0}}-1 < c'_{\alpha_{0}}$. Due to the existence of $F'_{1}, \ldots, F'_{k}$, by (\ref{19}), for $I \supseteq I_{\alpha_{0}} $ and $\mathcal{F} \in D(V)$,
\begin{equation}\label{25}
H(I,\mathcal{F}) \leq \sum_{I_{\alpha}\subseteq I, \alpha \neq \alpha_{0} }(c'_{\alpha} - \sum_{i \in I_{\alpha}}d_{F_{i}}^{+}(x))+ c'_{\alpha_{0}}-1-\sum_{i \in I_{\alpha_{0}}}d_{F_{i}}^{+}(x)< \sum_{I_{\alpha}\subseteq I }(c'_{\alpha} - \sum_{i \in I_{\alpha}}d_{F_{i}}^{+}(x)).
\end{equation}
By Theorem~\ref{1}, $F_{1}, \ldots, F_{k}$ can be completed to $F''_{1}, \ldots, F''_{k}$
%respectively
such that $\sum_{i \in I_{\alpha}} d^{+}_{F''_{i}}(x) \geq c_{\alpha}$ for $\alpha \in [l]$. Since $\sum_{i \in I_{\alpha_{0}}} d^{+}_{F''_{i}}(x) \geq c_{\alpha_{0}}> \sum_{i \in I_{\alpha_{0}}}d^{+}_{F_{i}}(x) $, there exists $e_{0} \in E^{+}_{F''_{i_{0}} \setminus F_{i_{0}}}(x)$ for some $i_{0} \in I_{\alpha_{0}}$. Update $F_{i_{0}}:=F_{i_{0}}+e_{0}$. By the necessity of Theorem~\ref{1}, (\ref{4})
%still
holds; due to Observation~\ref{21} and (\ref{25}), (\ref{22}) still holds. Thus we can continue to add edges to $F_{i}$ for $i \in I_{\alpha_{0}}$ such that $\sum_{i \in I_{\alpha_{0}}}d^{+}_{F_{i}}(x)$ increases until $\tau$ is decreased. This finishes the induction step.

\section{Remarks}

The following framework on bipartite graphs and supermodular functions is
due to Lov\'asz \cite{L-30}, 
it was extended by Frank and Tardos \cite{frank-T-22}, and
very recently B\'{e}rczi and Frank \cite{berczi1,berczi2,berczi3}  have made quite some further developments on it.
In this section, we  integrate our work of arborescence augmentation to this well-studied framework, and present some more generalized forms of our work.

Let $G =(S,T;E)$ be a bipartite graph with bipartition $S \cup T$ and edge set $E$. For $X \subseteq T$, let
\[
\Gamma_{G}(X) = \{s \in S: \mbox{ there~is~an~edge } st \in E ~\mbox{with~some } t \in X \}.
\]
We say that $G$ covers a set function $p_{T}$ on $T$ if $|\Gamma_{G} (X)| \geq p_{T} (X)$  for $ \emptyset \neq X \subseteq T$. Denote by $E^{*}$ the edge set of the complete bipartite graph with bipartition $S \cup T$.

Let $D=(V+x, A)$ be a digraph, $F_{1}, \ldots, F_{k}$ be arc disjoint $x$-arborescences in $D$. Let $G_{0}=([k], V; E_{0})$ be a bipartite graph and $iv \in E_{0}$ if $v \in V(F_{i})$ for $i \in [k]$ and $v \in V$.

\begin{lem} \label{53}
	$F_{1}, \ldots, F_{k}$ can be completed to arc disjoint spanning $x$-arborescences $F^{*}_{1},$ $ \ldots, $ $ F^{*}_{k}$ %respectively
	if and only if there exists $E \subseteq E^{*}$ such that
	\begin{itemize}
		\item [(i)] the bipartite graph $G^{+}=([k], V; E_{0} \cup E) $ is simple and covers $ k-d_{{A \setminus (\cup_{i=1}^{k}F_{i} \cup E^{+}(x))}}^{-}$;
		\item [(ii)] $d_{E}(v) \leq w_{[k]}(v) $ for $v \in V$ ($d_{E}(v)$ %denote
		is the number of edges in $E$ incident with $v$).
	\end{itemize}
\end{lem}
\begin{proof}
	For the necessity, let $E \subseteq E^{*}$ such that $iv \in E$ if and only if $v \in N_{F^{*}_{i}}^{+}(x)\setminus V(F_{i})$ for $i \in [k]$ and $v \in V$. Clearly, $d_{E}(v) \leq w_{[k]}(v)$ for $v \in V$, $E \cap E_{0} = \emptyset$ and $iv \in E_{0} \cup E$ if and only if $v \in V(F_{i}) \cup N^{+}_{F^{*}_{i}}(x)$. Thus $\Gamma_{G^{+}}(X)=\{i \in [k]: X \cap (V(F_{i}) \cup N^{+}_{F^{*}_{i}}(x)) \neq \emptyset \}$ for $\emptyset \neq X \subseteq V$. Note that $d_{{A \setminus (\cup_{i=1}^{k}F_{i} \cup E^{+}(x))}}^{-}$ is the in-degree function of $ D-\cup_{i=1}^{k}F_{i}-x$. If we see $F^{*}_{i}-F_{i}-x$ as a branching with root set $V(F_{i}) \cup N^{+}_{F^{*}_{i}}(x)$ in $D-\cup_{i=1}^{k}F_{i}-x$ for $1 \leq i \leq k $, by Theorem \ref{16}, for $\emptyset \neq X \subseteq V$, $ d_{{A \setminus (\cup_{i=1}^{k}F_{i} \cup E^{+}(x))}}^{-}(X) \geq |\{i \in [k]: X \cap (V(F_{i}) \cup N^{+}_{F^{*}_{i}}(x)) =\emptyset \} |$, this gives    $|\Gamma_{G^{+}}(X)| \geq k-d_{{A \setminus (\cup_{i=1}^{k}F_{i} \cup E^{+}(x))}}^{-}(X)$.
	
	Next, we prove the sufficiency. By $(i)$, $|\Gamma_{G^{+}}(X)| \geq k-d_{{A \setminus (\cup_{i=1}^{k}F_{i} \cup E^{+}(x))}}^{-}(X)$ for $\emptyset \neq X \subseteq V$. Then $|\Gamma_{G^{+}}(V)| \geq  k-d_{{A \setminus (\cup_{i=1}^{k}F_{i} \cup E^{+}(x))}}^{-}(V) = k $, so $N_{G^{+}}(i) \neq \emptyset$ for $i \in [k]$; $d_{{A \setminus (\cup_{i=1}^{k}F_{i} \cup E^{+}(x))}}^{-}(X) \geq |[k] \setminus \Gamma_{G^{+}}(X)|   =|\{i \in [k]: X \cap N_{G^{+}}(i) = \emptyset \}|$ for $\emptyset \neq X \subseteq V$.
	By Theorem \ref{16}, there exist arc disjoint branchings $B_{i}$ with root set $ N_{G^{+}}(i) \supseteq V(F_{i})-x$ in $D - \cup_{i=1}^{k}F_{i}-x$ for $1 \leq i \leq k$. Since $G^{+}$ is simple, $ N_{G^{+}}(v) = \{i \in [k]: v \in V(F_{i}) \} \dot{\cup} \{i \in [k]: iv \in E \}$ for $v \in V$. So $d_{E}(v)=|\{i:v \in N_{G^{+}}(i) \setminus V(F_{i}) \}|$.
	$(ii)$ implies for $v \in V$, $|\{i:v \in N_{G^{+}}(i) \setminus V(F_{i}) \}|=d_{E}(v) \leq w_{[k]}(v) \leq [x, v]_{A \setminus \cup_{i=1}^{k}F_{i}}$. Thus there exist disjoint arc sets $E_{i} \subseteq E^{+}_{A \setminus \cup_{i=1}^{k}F_{i}}(x)$ such that $E_{i}$ consists of arcs from $x$ to each vertex in $N_{G^{+}}(i) \setminus V(F_{i})$ for $1 \leq i \leq k$. $B_{i} \cup F_{i} \cup E_{i}$ is the spanning $x$-arborescence as demanded for $1 \leq i \leq k$.
\end{proof}

Next we show that Theorem~\ref{10} has an equivalent form, which can be obtained by replacing  $ d_{{A \setminus \cup_{i=1}^{k}F_{i}}}^{- }(X) $ with $ d_{{A \setminus (\cup_{i=1}^{k}F_{i} \cup E^{+}(x))}}^{-}(X) + \widetilde{w}_{[k]}(X)$ in (\ref{11}).

\noindent \textbf{Theorem \ref{10}'} Let $D=(V+x, A)$ be a digraph, $F_{1}, \ldots, F_{k}$ be $k$ arc-disjoint $x$-arborescences. They can be completed to $k$ arc-disjoint spanning $x$-arborescences
% $F^{*}_{1}, \ldots, F^{*}_{k}$
if and only if for any $\emptyset \neq X \subseteq V$,
\begin{equation}\label{54-R}
d_{{A \setminus (\cup_{i=1}^{k}F_{i} \cup E^{+}(x))}}^{-}(X) + \widetilde{w}_{[k]}(X)\geq |P(X)|.
\end{equation}
\begin{proof}
	For sufficiency, since $d_{{A \setminus \cup_{i=1}^{k}F_{i}}}^{-}(X)=d_{{A \setminus (\cup_{i=1}^{k}F_{i} \cup E^{+}(x))}}^{-}(X)+[x,X]_{A \setminus \cup_{i=1}^{k}F_{i}} \geq d_{{A \setminus (\cup_{i=1}^{k}F_{i} \cup E^{+}(x))}}^{-}(X) + \widetilde{w}_{[k]}(X)$ for $ \emptyset \neq X \subseteq V $, (\ref{54-R}) implies (\ref{11}). By Theorem~\ref{10}, $F_{1}, \ldots, F_{k}$ can be completed to  $k$ arc-disjoint spanning $x$-arborescences. % spanning.
	
	Next, we prove the necessity. Note that for $i \in [k]$ and $v \in V$, $[x,v]_{F^{*}_{i} \setminus F_{i}} \leq 1$ and the equality holds only if $v \notin V(F_{i})$. So for $v \in V$, $[x,v]_{\cup_{i=1}^{k}F_{i}^{*} \setminus F_{i}} \leq \min \{|\{i \in [k]: v \notin V(F_{i}) \}|, [x,v]_{A \setminus \cup_{i=1}^{k} F_{i}} \}= \widetilde{w}_{[k]}(v)$.
	
	Let $ \emptyset \neq X \subseteq V$.
	Since $\cup_{i=1}^{k}F^{*}_{i} \setminus F_{i}= E^{+}_{\cup_{i=1}^{k}F_{i}^{*} \setminus F_{i}}(x) \dot{\cup} ((\cup_{i=1}^{k}F^{*}_{i} \setminus F_{i}) \setminus E^{+}(x))$  	
	and $((\cup_{i=1}^{k}F^{*}_{i} \setminus F_{i}) \setminus E^{+}(x)) \subseteq A \setminus (\cup_{i=1}^{k}F_{i} \cup E^{+}(x ) ) $,
	\begin{align*}
	d_{\cup_{i=1}^{k}F^{*}_{i} \setminus F_{i}}^{-}(X)-d_{{A \setminus (\cup_{i=1}^{k}F_{i} \cup E^{+}(x))}}^{-}(X)
	& \leq d_{ E^{+}_{\cup_{i=1}^{k}F_{i}^{*} \setminus F_{i}}(x) }^{-}(X)   	
	= [x,X]_{\cup_{i=1}^{k}F^{*}_{i} \setminus F_{i}} \\ 	
	& = \sum_{v \in X} [x,v]_{\cup_{i=1}^{k}F_{i}^{*} \setminus F_{i}} \leq \sum_{v \in X}w_{[k]}(v) = \widetilde{w}_{[k]}(X). 	
	\end{align*}
	
	Note that for $i \in [k]$, $X \cap V(F_{i}) =\emptyset$ implies $d_{F_{i}^{*} \setminus F_{i}}^{-}(X) \geq 1$. Hence, $d^{-}_{\cup_{i=1}^{k}F_{i}^{*} \setminus F_{i}}(X) \geq |P(X)|$, 
	and then
	$d_{{A \setminus (\cup_{i=1}^{k}F_{i} \cup E^{+}(x))}}^{-}(X) + \widetilde{w}_{[k]}(X) \geq d_{\cup_{i=1}^{k}F_{i}^{*} \setminus F_{i}}^{-}(X)\geq |P(X)|$.
\end{proof}

%Note that
Since $|\Gamma_{G_{0}}(X)| = k- |P(X)| $,
(\ref{54-R}) is equivalent to
\begin{equation}\label{43}
|\Gamma_{G_{0}}(X)|+\widetilde{w}_{[k]}(X) \geq k-d_{{A \setminus (\cup_{i=1}^{k}F_{i} \cup E^{+}(x))}}^{-}(X).
\end{equation}
It follows from Lemma~\ref{53} and Theorem~\ref{10}' that there exists $E \subseteq E^{*}$ such that
\begin{itemize}
	\item [(i)]
	%$(i)$
	the graph $G^{+}=([k], V; E_{0} \cup E) $ is simple and covers $ k-d_{{A \setminus (\cup_{i=1}^{k}F_{i} \cup E^{+}(x))}}^{-}$,
	%and
	% $(ii)$
	\item [(ii)] $d_{E}(v) \leq w_{[k]}(v) $ for $v \in V$;
\end{itemize}
if and only if for $\emptyset \neq X \subseteq V$, (\ref{43}) holds.

%Corresponding to the above upper bound $w_{[k]}$, and special positively %intersecting supermodular function $k-d_{2}^{-}$, we have the following %general form.

The next theorem is a more generalized version of the above relations: since $k-d_{{A \setminus (\cup_{i=1}^{k}F_{i} \cup E^{+}(x))}}^{-}$ is a special $p_T$,
%positively intersecting supermodular function,
and  $w_{[k]}$ is a special case of function $g$ ($g$ is defined next in the theorem).

\begin{thm}
	Let $G_{0}=(S, T;E_{0})$ be a simple bipartite graph and $p_{T}$ be a positively intersecting supermodular function on $2^T$ such that $ p_{T} \leq |S|$. Let $g: T \rightarrow \mathbb{N}$. Then there exists $E \subseteq E^{*}$ such that
	\begin{itemize}
		\item [(i)] %(i)
		the graph $G^{+}=(S,T; E \cup E_{0})$ is simple and covers $p_{T}$, % and
		\item [(ii)] %(ii)
		$d_{E}(t) \leq g(t)$ for $t \in T$;
	\end{itemize}
	if and only if for $\emptyset \neq T_{0}\subseteq T$,
	\begin{equation}\label{44}
	|\Gamma_{G_{0}}(T_{0})|+\widetilde{g}(T_{0}) \geq p_{T}(T_{0}).
	\end{equation}
\end{thm}

%\proof

\begin{proof}
	The necessity is obvious.
	We just prove the sufficiency. If $g \equiv 0$, then let $E := \emptyset$, we are done. %Otherwise, s
	Suppose $ g(t_{0}) >0$ for some $t_{0} \in T$. If for any $T_{0} \subseteq T$, $|\Gamma_{G_{0}}(T_{0})|+\widetilde{g}(T_{0}) > p_{T}(T_{0})$, update  $g(t_{0}):=g(t_{0})-1$ and (\ref{44}) still holds. Otherwise, choose a maximal $T_{1} \subseteq T$ such that $t_{0} \in T_{1}$ and $|\Gamma_{G_{0}}(T_{1})|+\widetilde{g}(T_{1}) = p_{T}(T_{1})$.  We  show next that if
	$t_{0} \in T_{0} \subseteq T$
	% for $T_{0} \subseteq T$ such that $t_{0} \in T_{0}$
	and $|\Gamma_{G_{0}}(T_{0})|+\widetilde{g}(T_{0}) = p_{T}(T_{0})$, then $T_{0} \subseteq T_{1}$.  Suppose to the contrary that there exists $T_{2} \subseteq T$, $t_{0} \in T_{2}$, $T_{2} \setminus T_{1} \neq \emptyset$ and $|\Gamma_{G_{0}}(T_{2})|+\widetilde{g}(T_{2}) = p_{T}(T_{2}) $. Since $g(t_{0})>0$, $t_{0} \in T_{i}$ and $|\Gamma_{G_{0}}(T_{i})|+\widetilde{g}(T_{i}) = p_{T}(T_{i})$ for $i=1,2$, $p_{T}(T_{1}), p_{T}(T_{2}) >0$. Note that $|\Gamma_{G_{0}}|$, and $\widetilde{g}$ are intersecting submodular and $p_{T}$ is positively intersecting supermodular on $2^{T}$,
	we have
	\[
	\begin{split}
	& 0=|\Gamma_{G_{0}}(T_{1})|+\widetilde{g}(T_{1})-p_{T}(T_{1})
	+|\Gamma_{G_{0}}(T_{2})|+\widetilde{g}(T_{2})-p_{T}(T_{2}) \geq \\
	& |\Gamma_{G_{0}}(T_{1} \cup T_{2})|+\widetilde{g}(T_{1} \cup T_{2})-p_{T}(T_{1} \cup T_{2})+
	|\Gamma_{G_{0}}(T_{1} \cap T_{2})|+\widetilde{g}(T_{1}\cap T_{2})-p_{T}(T_{1} \cap T_{2}) \\
	&\geq 0;
	\end{split}
	\]
	by (\ref{44})$, |\Gamma_{G_{0}}(T_{1} \cup T_{2})|+\widetilde{g}(T_{1} \cup T_{2}) = p_{T}(T_{1} \cup T_{2})$, contradicting the maximality of $T_{1}$. Since $|\Gamma_{G_{0}}(T_{1})|= p_{T}(T_{1})-\widetilde{g}(T_{1}) < |S| $, there exists $s_{0} \in S \setminus \Gamma_{G_{0}}(T_{1})$. Update $g(t_{0}):=g(t_{0})-1$ and $E := E+s_{0}t_{0}$.
	%and $E_{0}:=E_{0}+s_{0}t_{0}$.
	Then (\ref{44}) still holds. Continue the above process  until $g \equiv 0$ and $E$ is the set of edges added to $E_{0}$.
	%\QEDA
\end{proof}

Next we %shall
present a more generalized form of our main theorem (Theorem~\ref{32}). 

Assume $F_{1}, \ldots, F_{k}$ can be completed to arc disjoint spanning $F^{*}_{1}, \ldots, F^{*}_{k}$. % respectively.
Define $E \subseteq E^{*}$ as $iv \in E$ if and only if $v \in N_{F^{*}_{i}}^{+}(x)\setminus V(F_{i})$ for $i \in [k]$ and $v \in V$. Note that $ d_{E}(i)=d_{F^{*}_{i}}^{+}(x)- d_{F_{i}}^{+}(x)$ for $i \in [k]$. Let $I_{1}, \ldots, I_{l}$ be a partition of $[k]$. Let $d_{1},\ldots, d_{l}$ and $d'_{1}, \ldots, d'_{l}$ be nonnegative integers such that $d_{\alpha} \leq d'_{\alpha}$ for $1 \leq \alpha \leq l$.

In Theorem~\ref{32}, let $c_{\alpha}=d_{\alpha}+ \sum_{i \in I_{\alpha}}d_{F_{i}}^{+}(x)$ and $c'_{\alpha}=d'_{\alpha}+ \sum_{i \in I_{\alpha}}d_{F_{i}}^{+}(x)$.  
Then Theorem~\ref{32} gives a characterization for the existence of  $\{F_{i}^{*} \}_{i=1}^{k}$ such that $ d_{\alpha} \leq \sum_{i \in I_{\alpha}} ( d_{F^{*}_{i}}^{+}(x)- d_{F_{i}}^{+}(x) ) \leq d'_{\alpha}$ for $\alpha \in [l]$, that is $d_{\alpha} \leq \sum_{i \in I_{\alpha}} d_{E}(i) \leq d'_{\alpha} $.
%
%Similarly to %the arguments in
%the proof of Theorem~\ref{10}',
%
%the following updates need to be made:
%
Apply Theorem~\ref{32} with the above $c_{\alpha}$ and $c'_{\alpha}$, then the related formulas are updated as following:
%\red{(the check of these is similar to the proof of Theorem~\ref{10}')}:
\begin{itemize}
	\item  replace $ d_{{A \setminus \cup_{i=1}^{k}F_{i}}}^{- }(X) $  by $ d_{{A \setminus (\cup_{i=1}^{k}F_{i} \cup E^{+}(x))}}^{-}(X) + \widetilde{w}_{[k]}(X)$  in (\ref{4})
	(checking similarly to Theorem~\ref{10}'); 
	
	\item  replace $c_{\alpha}-\sum_{i \in I_{\alpha}}d_{F_{i}}^{+}(x)$  by $d_{\alpha}$ in (\ref{4}) and (\ref{27});
	
	\item replace $c'_{\alpha}-\sum_{i \in I_{\alpha}}d_{F_{i}}^{+}(x)$ by $d'_{\alpha}$ in (\ref{22}).
\end{itemize}
For $\emptyset \neq X \subseteq V$, we have $P_{I}(X) \dot{\cup } (\overline{I} \cup \Gamma_{G_{0}}(X))=[k]$, $|P_{I}(X)|=k-|\overline{I} \cup \Gamma_{G_{0}}(X)|$.
Then we obtain the following inequalities from (\ref{4}), (\ref{27}) and (\ref{22}) respectively:
\begin{equation} \label{36}
\sum^{t}_{j=1} (|\overline{I} \cup \Gamma_{G_{0}}(X_{j})| +\widetilde{w}_{[k]}(X_{j}) )\geq  \sum^{t}_{j=1}(k-d^{-}_{{A \setminus (\cup_{i=1}^{k}F_{i} \cup E^{+}(x))}} (X_{j})) +
\sum_{ I_{\alpha} \subseteq  \overline{I}}d_{\alpha}- \widetilde{w}_{\overline{I}}(V \setminus \cup_{j=1}^{t}X_{j}),
\end{equation}
\begin{equation}\label{37}
\widetilde{w}_{\overline{I}}(V) \geq \sum_{I_{\alpha} \subseteq \overline{I}} d_{\alpha},
\end{equation}
\begin{equation}\label{38}
\sum_{j=1}^{t}(k - d_{{A \setminus (\cup_{i=1}^{k}F_{i} \cup E^{+}(x))} }^{-}(X_{j}) - |\overline{I} \cup \Gamma_{G_{0}}(X_{j})|) \leq \sum_{I_{\alpha} \subseteq I} d'_{\alpha}.
\end{equation}

Then Theorem~\ref{32} gives the following:  there exists $E \subseteq E^{*}$ such that
	$(i)$ the bipartite graph $G^{+}=([k], V; E_{0} \cup E) $ is simple and covers $ k-d_{{A \setminus (\cup_{i=1}^{k}F_{i} \cup E^{+}(x))}}^{-}$,
	$(ii)$ $d_{\alpha} \leq \sum_{i \in I_{\alpha}} d_{E}(i) \leq d'_{\alpha}$ for $\alpha \in [l]$,
	$(iii)$ $d_{E}(v) \leq w_{[k]}(v)$ for $v \in V$; 	
%\end{itemize}
if and only if for any disjoint $X_{1}, \ldots, X_{t} \subseteq V$ and $I \subseteq [k]$ that is the union of some of $I_{1}, \ldots, I_{l}$, 
%\begin{itemize}
	%\item[(i)]
		$(i)$ (\ref{36}) holds; in particular, when $t=0$, (\ref{36}) implies (\ref{37}) holds;
	%\item[(ii)]
		$(ii)$  (\ref{38}) holds.
%\end{itemize}

The next theorem will generalize the above relations: since $k-d_{{A \setminus (\cup_{i=1}^{k}F_{i} \cup E^{+}(x))}}^{-}$ is a special $p_T$, and  $w_{I}$ is a special %case of function
$g_{S_{0}}$ (where $g_{S_{0}}$ is defined next in the theorem; for the similarities, %between these two functions,
note that by definition, $w_{I}(u) = \min\{ |I \setminus \Gamma_{G_{0}}(u)|, w_{[k]}(u) \}$ for $u \in V$ and $I \subseteq [k]$).

\begin{thm}\label{41}
	Let $G_{0}=(S, T;E_{0})$ be a simple bipartite graph, $p_{T}$ be a positively intersecting supermodular function on $T$ such that $ p_{T} \leq |S|$. Let $ \{S_{1}, \ldots, S_{l} \}$ be a partition of $S$.
	Suppose $f_{1}, f_{2}: [l] \rightarrow \mathbb{N}$ satisfy that $f_{1} \leq f_{2}$ and $g: T \rightarrow \mathbb{N}$. Then there exists $E \subseteq E^{*}$ such that 
	\begin{itemize}
		\item[(i)] the graph $G^{+}=(S,T; E \cup E_{0})$ is simple and covers $p_{T}$,
		\item[(ii)] $f_{1}(\alpha) \leq \sum_{s \in S_{\alpha}}d_{E}(s) \leq f_{2}(\alpha)$ for $\alpha \in [l]$,
		\item[(iii)] $d_{E}(t) \leq g(t)$ for $t \in T$;
	\end{itemize}
	if and only if for any disjoint $T_{1}, \ldots, T_{n} \subseteq T$ and any $S_{0} \subseteq S$ that is the union of some of $S_{1}, \ldots, S_{l}$,  	 
	%\noindent
	\begin{itemize}
		\item[(i)] \begin{equation}\label{40}
		\sum^{n}_{j=1} (|\overline{S_{0}} \cup \Gamma_{G_{0}}(T_{j})| +\widetilde{g}(T_{j}) )\geq  \sum^{n}_{j=1}p_{T}(T_{j}) +
		\sum_{ S_{\alpha} \subseteq  \overline{S_{0}}}f_{1}(\alpha)- \widetilde{g}_{\overline{S_{0}}}(T \setminus \cup_{j=1}^{n}T_{j}),
		\end{equation}
		where $\overline{S_{0}}=S \setminus S_{0}$ and $g_{S_{0}}(t):= \min\{|S_{0} \setminus \Gamma_{G_{0}}(t)|, g(t) \}$ for $t \in T$.
		
		In particular, when $n=0$, (\ref{40}) implies
		\begin{equation}\label{88}
		\widetilde{g}_{\overline{S_{0}}}(T) \geq \sum_{S_{\alpha} \subseteq \overline{S_{0}}}f_{1}(\alpha).
		\end{equation} 	
		%\noindent
		\item[(ii)] \begin{equation}\label{42}
		\sum_{j=1}^{n}(p_{T}(T_{j}) - |\overline{S_{0}} \cup \Gamma_{G_{0}}(T_{j}) |) \leq \sum_{S_{\alpha} \subseteq S_{0}} f_{2}(\alpha)
		\end{equation} 	
	\end{itemize}
\end{thm}

Let $M_{S}=(S, r_{S})$ be a matroid on $S$ with rank function $r_{S}$. A bipartite graph $G=(S, T ; E)$ is said to $M_{S}$-cover  a set function $p_{T}$ on $T$ if for $\emptyset \neq T_{0} \subseteq T$,
$r_{S}(\Gamma_{G}(T_{0})) \geq p_{T}(T_{0})$.
In (\ref{40}) and (\ref{42}), if we replace $|\overline{S_{0}} \cup \Gamma_{G_{0}}(T_{j}) |$with $r_{S}(\overline{S_{0}} \cup \Gamma_{G_{0}}(T_{j}) )$,  then Theorem~\ref{41} will have a strengthened  matroidal version:

Let $S, T$ be the above disjoint sets and $I(G;S_{0}, T_{0}):=r_{S}(\overline{S_{0}} \cup \Gamma_{G}(T_{0}) )$, where
$G=(S, T ; E)$,
%
%$G$ is a bipartite graph with partition $S\cup T$,
$S_{0} \subseteq S$ and $T_{0}\subseteq T$. Note that $I$ has the following properties:
\begin{itemize}
	\item[(i)] $I(G_{1}; S_{0}, T_{0}) \leq I(G_{2}, S_{0}, T_{0})$ if $E(G_{1}) \subseteq E(G_{2})$;
	
	\item[(ii)] $I(G; S_{0}, T_{0}) \geq I(G; S'_{0}, T_{0})$ if $S_{0} \subseteq S'_{0}$ and $I(G; S, T_{0})=r_{S}(\Gamma_{G}(T_{0}))$;
	
	\item[(iii)] $I(G; S_{0}, T_{0})$ is submodular on $2^{T}$.
\end{itemize}
Thanks to these three properties, the proof of Theorem~\ref{41} and its matroidal version is (exactly) the same as the proof of Theorem~\ref{32},
the check for details is skipped here. %left to the reader.

As an application, we note that  the matroidal version of Theorem~\ref{41} implies the following result of B\'{e}rczi and Frank \cite[Theorem 3]{berczi2}. To see this,
let $l=|S|$ and $ \{S_{1}, \ldots, S_{|S|} \}$ be a partition of $S$ with each part a singleton.     Let $f_{1} \equiv 0$, $g \equiv |S|$ and $f_{2}=f $, then (\ref{40}) and (\ref{88}) hold in the matroidal version of Theorem~\ref{41} (by replacing $|\overline{S_{0}} \cup \Gamma_{G_{0}}(T_{j})|$ with $r_{S}(\overline{S_{0}} \cup \Gamma_{G_{0}}(T_{j}) )$). And the matroidal version of Theorem~\ref{41} characterizes the $M_{S}$-covering problem with $d_{E}(s) \leq f(s)$ for $s \in S$. 

\begin{thm}(\cite{berczi2}) \label{55-B}
	Let $G_{0}=(S, T;E_{0})$ be a simple bipartite graph and $p_{T}$ be a positively intersecting supermodular function on $T$ such that $p_{T} \leq r_{S}(S)$. Let $M_{S}=(S, r_{S})$ be a matroid on $S$ with rank function $r_{S}$. Let $f: S \rightarrow \mathbb{N}$. Then there exists $E \subseteq E^{*}$ such that  	
	\begin{itemize}
		\item[(i)] the graph $G^{+}=(S,T; E \cup E_{0})$ is simple and $M_{S}$-covers $p_{T}$,
		\item[(ii)] $d_{E}(s)=f(s)$ for $s \in S$;
	\end{itemize}
	if and only if
	\begin{itemize}
		\item[(i)] for any $s \in S$, $d_{E_{0}}(s) + f(s) \leq |T|$;
		
		\item[(ii)] for any disjoint $T_{1}, \ldots, T_{n} \subseteq T$ and $S_{0} \subseteq S$,
		\begin{equation}\label{56}
		\sum_{j=1}^{n}(p_{T}(T_{j}) - r_{S}(\overline{S_{0}} \cup \Gamma_{G_{0}}(T_{j}) )) \leq \widetilde{f}(S_{0}).
		\end{equation}
	\end{itemize}
\end{thm}

The following question is interesting to us: Is there some relationship between Theorem \ref{Cai-21} and our main result Theorem \ref{32}? We tend to think Theorem \ref{32} will 
derive Theorem \ref{Cai-21}, but this is open.

\noindent {\bf Acknowledgements:} 
This paper is dedicated to Professor Xuding Zhu on the occasion of his 60th birthday. 
We thank two anonymous referees for their very detailed and helpful comments that help improved the presentation of this paper.

\end{document}